\documentclass[reqno]{amsart}

\usepackage{amscd,amsmath,amssymb,amsfonts}
\usepackage{mathrsfs}
\usepackage[dvips]{color} 
\usepackage[all]{xy}

\numberwithin{equation}{section}

\theoremstyle{plain}
    \newtheorem{thm}{Theorem}[section]
    \newtheorem{lem}[thm]{Lemma}
    
    \newtheorem{prop}[thm]{Proposition}
    \newtheorem{cor}[thm]{Corollary}
    \newtheorem{prob}[thm]{Problem}
    
    \newtheorem{conj}[thm]{Conjecture}
\theoremstyle{definition}    
    
    \newtheorem{exmp}[thm]{Example}
    \newtheorem{rem}[thm]{Remark}

\def\rank{\mathrm{rank}}

\def\Coker{\mathrm{Coker}}

\def\dlog{{\mathrm{dlog}}}
\def\dR{{\mathrm{d\hspace{-0.2pt}R}}}            % de Rham
\def\id{{\mathrm{id}}}              % identity
\def\Image{{\mathrm{Im}}}        % image
\def\Hom{{\mathrm{Hom}}}  
\def\Ext{{\mathrm{Ext}}}

\def\ker{{\mathrm{Ker}}}          % kernel
\def\Pic{{\mathrm{Pic}}}

\def\NS{{\mathrm{NS}}}

\def\End{{\mathrm{End}}}

\def\ord{{\mathrm{ord}}}

\def\reg{{\mathrm{reg}}}          %
\def\Res{\mathrm{Res}}
\def\Spec{{\mathrm{Spec}}}     % spectrum

\def\rank{{\mathrm{rank}}}

\def\bA{{\mathbb A}}

\def\C{{\mathbb C}}

\def\P{{\mathbb P}}

\def\Q{{\mathbb Q}}

\def\R{{\mathbb R}}

\def\Z{{\mathbb Z}}

\def\cH{{\mathscr H}}
\def\cM{{\mathscr M}}
\def\cD{{\mathscr D}}

\def\O{{\mathscr O}}

\def\vg{\varGamma}
\def\ve{\varepsilon}
\def\lra{\longrightarrow}

\def\hra{\hookrightarrow}

\def\ot{\otimes}
\def\op{\oplus}
\def\wt#1{\widetilde{#1}}

\def\ol#1{\overline{#1}}

\def\os#1#2{\overset{#1}{#2}}

\def\l{\lambda}
\def\Ev{\mathrm{Ev}}

\begin{document}
\title{Regulators of $K_2$ of Hypergeometric Fibrations}
\author{Masanori Asakura}
\address{Department of Mathematics, Hokkaido University, Sapporo, 060-0810 Japan}
\email{asakura@math.sci.hokudai.ac.jp}
\date{2016}

\maketitle

%%%%%%%%%%%%%%%%%%%%%%%%%%%%%%%%%%%%%%%%%%%%%%%%%
\section{Introduction}\label{intro-sect}
In the paper \cite{a-o-2}, Otsubo and the author introduced 
a certain class of fibrations of
algebraic varieties which
we named {\it hypergeometric fibrations} (abbreviated HG fibrations,
see \S \ref{HG-defn} for the definition).
In a series of our joint papers \cite{a-o-1}\ldots\cite{a-o-log}, we studied
$K_1$ of HG fibrations and the {\it Beilinson regulators}. 
Our main results are to describe
the regulators via the generalized hypergeometric functions 
\[
{}_pF_{p-1}\left({a_1,\ldots,a_p\atop b_1,\ldots,b_{p-1}};x\right)
=\sum_{n=0}^\infty\frac{(a_1)_n\cdots(a_p)_n}{(b_1)_n\cdots(b_{p-1})_n}\frac{x^n}{n!},
\quad (\alpha)_n:=\Gamma(\alpha+n)/\Gamma(\alpha)
\]
(we refer \cite{Bailey}, \cite{bateman} or \cite{slater} to the reader for the fundamental 
theory on hypergeometric functions).

In this paper we study $K_2$ of HG fibrations.
Let $f:X\to \P^1$ be a HG fibration defined in \S \ref{HG-defn}, and
$X_t=f^{-1}(t)$ a smooth fiber, then we discuss the Beilinson regulator map
\[
\reg:K_2(X_t)\lra H^2_\cD(X_t,\Z(2))
\]
to the Deligne-Beilinson cohomology group (e.g. \cite{schneider}).
We shall discuss the following cases.
\begin{itemize}
\item
$f$ is of Fermat type given in \S \ref{Fermat-sect},
\item
$f$ is of Gauss type given in \S \ref{Gauss-sect}, 
\item 
$f$ is an elliptic fibration (e.g. the Legendre family).
\end{itemize}
In the above cases, there are nontrivial elements in $K_2(X_t)$.
The main theorems are to give explicit descriptions of the regulators by 
linear combinations of the hypergeometric functions of the following types
\[
{}_3F_2\left({a,a,a\atop b,a+1};x\right),
\quad {}_4F_3\left({a,b,1,1\atop 2,2,2};x\right).
\]
The precise formulas are given in Theorems \ref{m-fermat-thm1}, 
\ref{m-fermat-thm2}, \ref{m-fermat-thm3} in \S \ref{m-fermat-sect2}
for the Fermat type, and
in Theorem \ref{m-gauss-thm1}
in \S \ref{m-gauss-sect2} for the Gauss type.

In \S \ref{bei-sect}, we give similar formulas for some elliptic fibrations, 
such as the Legendre family.
With the aid of MAGMA,
we give a number of numerical verifications of Beilinson's conjecture on $L(E,2)$,
the $L$-function of an elliptic curve over $\Q$.
In \cite{RZ},
Rogers and Zudilin proved certain formulas which describes $L(E,2)$ by
special values of hypergeometric functions.
Applying their result, we can obtain a ``theorem'' on Beilinson's conjecture 
for an elliptic curve of conductor $24$. This seems a new approach toward the
the Beilinson conjecture for elliptic curves. The author hopes that,
the study of the Beilinson conjecture by the hypergeometric functions 
will be developed more and bring a new progress.

Finally we note that
there are previous works \cite{otsubo-1}, \cite{otsubo-2}
by Otsubo on hypergeometric functions and
regulators on $K_2$ of Fermat curves.
Our results and method are entirely different from his though
this paper is inspired a lot.

\medskip

The author expresses sincere gratitude to Professor Noriyuki Otsubo
for the stimulating discussion, especially on special values of
$L$-functions of elliptic curves.
He also expresses special thanks to Professor Wadim Zudilin for reading the first
draft carefully and providing the proof of 
Lemma \ref{m-fermat-thm4}.

\section{Hypergeometric Fibrations}\label{HG-sect}
Throughout this paper, we denote the fractional part of $x\in\Q$
by  $\{x\}$ :
\[
\{x\}:=x-\lfloor x\rfloor.
\]
The Gaussian hypergeometric function
\[
{}_2F_1\left({a,b\atop c};x\right)
\]
is simply written by $F(a,b,c;x)$.

\subsection{Definition}\label{HG-defn}
Let $R$ be a finite-dimensional semisimple commutative $\Q$-algebra.
Let $e:R\to E$ be a projection onto a number field $E$.
For a $R$-module $H$, we write \[H(e):=E\ot_{e,R}H,\] and call it the $e$-part of $H$.

Let $X$ be a projective smooth variety over a field $k$.
Let $f:X\to \P^1$ be a surjective morphism over $k$ which is smooth over 
$U\subset \P^1$.
Let $A=\Pic^0_f\to U$ be the Picard scheme over $U$.
We say $f$ is a {\it hypergeometric fibration with multiplication by $(R,e)$} 
(abbreviated HG fibration) if
it is endowed with a ring homomorphism (called a {\it multiplication} by $R$)
\[
R\lra \End_U(A)\ot\Q
\]
and the following conditions hold. We fix an inhomogeneous coordinate $t\in \P^1$.
\begin{itemize}
\item $f$ is smooth outside $t=0,1,\infty$, hence we may take 
$U=\P^1\setminus\{0,1,\infty\}$.
\item
Denote by $A(e)\to U$ the $e$-part of the abelian fibration
which corresponds to the $e$-part
$(R^1f_*\Q_l)(e)$ of a $l$-adic sheaf,
\[
T_l A(e)\ot\Q\cong R^1f_*\Q_l(e).
\]
Then $\dim(A(e)/U)=[E:\Q]$ or equivalently $\dim_{\Q_l}(R^1f_*\Q_l)(e)=2[E:\Q]$. 
\item
The abelian fibration $A(e)\to U$ has a totally degenerate semistable reduction at $t=1$.
\end{itemize}

The last condition is equivalent to say that
the local monodromy $T$ on $(R^1f_*\Q_l)(e)$ at $t=1$ is unipotent
and the rank of log monodromy $N:=\log(T)$ is maximal, namely 
$\rank(N)=\frac{1}{2}\dim_\Q (R^1f_*\Q_l)(e)$
($=[E:\Q]$ by the second condition).

\begin{exmp}[Elliptic fibrations]
The simplest example of hypergeometric fibrations is an elliptic fibration $f:X\to\P^1$ 
which satisfies that $f$ is smooth over $\P^1\setminus\{0,1,\infty\}$ and has a multiplicative reduction at $t=1$.
In this case we take $R=E=\Q$, $e=\id$.
We refer the reader to \cite{hirzebruch} for the classification of elliptic fibrations
over $\P^1$ which has singular fibers at most over 3 points.
In particular, there are 14 cases of such fibrations which have at least one multiplicative reduction.
\end{exmp}

\subsection{HG fibration of Fermat type}\label{Fermat-sect}
Suppose that the characteristic of $k$ is $0$. 
Let $f:X\to \P^1$ be the fibration such that the general fiber $X_t=f^{-1}(t)$ is a 
smooth projective curve defined by an equation 
\[(x^n-1)(y^m-1)=1-t,\quad n,m\geq 2.\]
We call $f$ a fibration of Fermat type\footnote{
The reason why we call ``Fermat type'' is that
the fiber over $t=0$ is
\[
(x^n-1)(y^m-1)=1\quad \Longleftrightarrow \quad x^{-n}+y^{-m}=1,
\] 
hence the Fermat curve appears in the degenerating fiber.}.
One can show
$g(X_t)=(n-1)(m-1)$ (e.g. by the Hurwitz formula).
Moreover $f$ is smooth outside $t=0,1,\infty$,
and $f$ has a 
totally degenerate semistable reduction at $t=1$. 
We denote by $\mu_k\subset \ol{k}^\times$ 
the group of $k$-th roots of unity.
Suppose $\mu_n,\mu_m\subset k^\times$.
The action $(x,y,t)\mapsto (\zeta_nx,\zeta_my,t)$ for $(\zeta_n,\zeta_m)\in\mu_n\times\mu_m$
gives a multiplication by the group ring $R=\Q[\mu_n\times\mu_m]$.
If $e:R\to E$ factors through
projections $\mu_n\times\mu_m\to\mu_n$ or $\mu_n\times\mu_m\to\mu_m$,
then $H^1(X_t)(e)=0$. Therefore
\[
H^1(X_t)=\bigoplus_e H^1(X_t)(e)
\]
where $e$ does not factor through
projections $\mu_n\times\mu_m\to\mu_n$ or $\mu_n\times\mu_m\to\mu_m$.
\begin{lem}\label{Fermat-lem1}
Put
\[
\omega_{i,j}:=x^{i-1}y^{j-1}\frac{m^{-1}dx}{y^{m-1}(x^n-1)}
=-x^{i-1}y^{j-1}\frac{n^{-1}dy}{x^{n-1}(y^m-1)}
\]
for $i,j\in\Z$.
Then 
$\vg(X_t,\Omega^1_{X_t})$ is $(n-1)(m-1)$-dimensional with basis
$\{\omega_{i,j}\mid 1\leq i\leq n-1,\,1\leq j\leq m-1\}$.
Hence 
\[
\dim_EH^1(X_t)(e)=\begin{cases}
0&\mbox{
$e$ factoring through
$\mu_n\times\mu_m\to\mu_n$ or $\mu_n\times\mu_m\to\mu_m$}\\
1&\mbox{others.}
\end{cases}
\]$f$ is a HG fibration with multiplication by $(R,e)$
if and only if $\dim_EH^1(X_t)(e)=1$, and then
\[
\vg(X_t,\Omega^1_{X_t})(e)=\bigoplus_{(i,j)\in I_e}k\cdot\omega_{i,j}
\]
\begin{equation}\label{Fermat-lem1-eq1}
I_e:=\{([si_0]_n,[sj_0]_m)\mid s\in(\Z/nm\Z)^\times\}
\end{equation}
where $(i_0,j_0)$ is a fixed index such that a homomorphism 
$R\to k$, $(\zeta_n,\zeta_m)\mapsto \zeta_n^{i_0}\zeta_m^{j_0}$ factors through $e$,
and $[a]_n$ denotes the unique integer such that $[a]_n\equiv a$ mod $n$ and
$0\leq [a]_n<n$.
\end{lem}
\begin{proof}
See \cite{a-o-2} \S 3.3.
\end{proof}
Suppose that the base field is $\C$. 
Let $\ve_1\in \mu_n$ and $\ve_2\in \mu_m$, and let $P(\ve_1,\ve_2)$ denotes the
singular point $(x,y)=(\ve_1,\ve_2)$ of $f^{-1}(1)$.
Let $\delta(\ve_1,\ve_2)\in H_1(X_t,\Z)$ be the vanishing cycle at $t=1$
which ``converges to
$P(\ve_1,\ve_2)$'', namely it is a homology cycle characterized by
\[
\frac{1}{(2\pi \sqrt{-1})^2}\oint_{t=1}\int_{\delta(\ve_1,\ve_2)}\omega
=\Res_P(\omega),\quad
\forall\,\omega\in H^2_\dR({\mathscr X}^*)
\]
where ${\mathscr X}^*$ is the tubular neighborhood of $f^{-1}(1)$ and
$\Res_P:H^2_\dR({\mathscr X}^*)\to\C$ is the Poincare residue map at $P=P(\ve_1,\ve_2)$.

For the later use, we here give a down-to-earth description of 
a path $\delta(\ve_1,\ve_2)$.
For $(\zeta_1,\zeta_2)\in\mu_n\times\mu_m$, we denote by $\sigma(\zeta_1,\zeta_2)$
the automorphisms of $X_t$ given by
$(x,y)\mapsto(\zeta_1 x,\zeta_2y)$.
Suppose $|t-1|\ll1$ and fix $\sqrt[n]{t}$. 
Let $Q_1(x,y)=(1,\infty)$ and $Q_t(x,y)=(\sqrt[n]{t},0)$ be points of $X_t$.
Define a (unique) path $u$ from $Q_t$ to $Q_1$ such that
the projection onto the $y$-plane is a line $\mathrm{arg}(y)=-\pi/m$
from $y=0$ to $y=\infty$. 
Put
\begin{equation}\label{Fermat-eq1}
\delta(1,1):=(1-\sigma(1,e^{\frac{2\pi \sqrt{-1}}{m}} ))u,\quad \delta(\ve_1,\ve_2)
:=\sigma(\ve_1,\ve_2)\delta(1,1).
\end{equation}

\medskip

%WinTpicVersion4.30c
{\unitlength 0.1in%
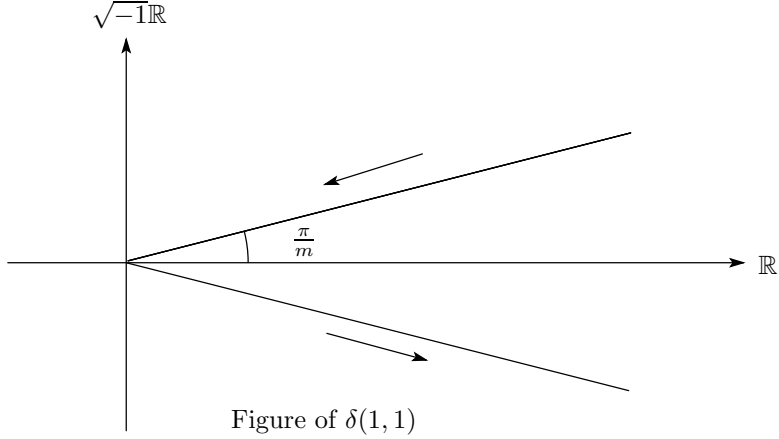
\begin{picture}( 51.7000, 29.9000)(  2.1000,-38.1000)%
% DOT 2 0 3 0 Black White
% 2 1590 3810 5380 1650
% 
\special{pn 4}%
\special{sh 1}%
\special{ar 1590 3810 8 8 0  6.28318530717959E+0000}%
\special{sh 1}%
\special{ar 5380 1650 8 8 0  6.28318530717959E+0000}%
% LINE 2 0 3 0 Black White
% 2 840 2170 3470 1500
% 
\special{pn 8}%
\special{pa 840 2170}%
\special{pa 3470 1500}%
\special{fp}%
% LINE 2 0 3 0 Black White
% 2 840 2170 3470 1500
% 
\special{pn 8}%
\special{pa 840 2170}%
\special{pa 3470 1500}%
\special{fp}%
% LINE 2 0 3 0 Black White
% 2 840 2170 3470 1500
% 
\special{pn 8}%
\special{pa 840 2170}%
\special{pa 3470 1500}%
\special{fp}%
% LINE 2 0 3 0 Black White
% 2 830 2180 3460 2850
% 
\special{pn 8}%
\special{pa 830 2180}%
\special{pa 3460 2850}%
\special{fp}%
% VECTOR 2 0 3 0 Black White
% 4 2380 1610 1870 1770 1880 2550 2400 2690
% 
\special{pn 8}%
\special{pa 2380 1610}%
\special{pa 1870 1770}%
\special{fp}%
\special{sh 1}%
\special{pa 1870 1770}%
\special{pa 1940 1769}%
\special{pa 1921 1754}%
\special{pa 1928 1731}%
\special{pa 1870 1770}%
\special{fp}%
\special{pa 1880 2550}%
\special{pa 2400 2690}%
\special{fp}%
\special{sh 1}%
\special{pa 2400 2690}%
\special{pa 2341 2653}%
\special{pa 2349 2676}%
\special{pa 2330 2692}%
\special{pa 2400 2690}%
\special{fp}%
% VECTOR 2 0 3 0 Black White
% 2 210 2180 4060 2180
% 
\special{pn 8}%
\special{pa 210 2180}%
\special{pa 4060 2180}%
\special{fp}%
\special{sh 1}%
\special{pa 4060 2180}%
\special{pa 3993 2160}%
\special{pa 4007 2180}%
\special{pa 3993 2200}%
\special{pa 4060 2180}%
\special{fp}%
% STR 2 0 3 0 Black White
% 4 4190 2090 4190 2190 5 0 0 0
% $\R$
\put(41.9000,-21.9000){\makebox(0,0){$\R$}}%
% STR 2 0 3 0 Black White
% 4 1380 2980 1380 3080 2 0 0 0
% Figure of $\delta(1,1)$
\put(13.8000,-30.8000){\makebox(0,0)[lb]{Figure of $\delta(1,1)$}}%
% VECTOR 2 0 3 0 Black White
% 2 830 3060 830 1010
% 
\special{pn 8}%
\special{pa 830 3060}%
\special{pa 830 1010}%
\special{fp}%
\special{sh 1}%
\special{pa 830 1010}%
\special{pa 810 1077}%
\special{pa 830 1063}%
\special{pa 850 1077}%
\special{pa 830 1010}%
\special{fp}%
% STR 2 0 3 0 Black White
% 4 650 850 650 950 2 0 0 0
% $\sqrt{-1}\R$
\put(6.5000,-9.5000){\makebox(0,0)[lb]{$\sqrt{-1}\R$}}%
% CIRCLE 2 0 3 0 Black White
% 4 828 2176 1448 2016 1468 2178 1442 2020
% 
\special{pn 8}%
\special{ar 828 2176 640 640  6.0343782  0.0031250}%
% STR 2 0 3 0 Black White
% 4 1760 2050 1760 2070 5 0 0 0
% $\frac{\pi}{m}$
\put(17.6000,-20.7000){\makebox(0,0){$\frac{\pi}{m}$}}%
\end{picture}}%

\vspace{-1.5cm}

\begin{lem}\label{Fermat-lem2}
\[
\int_{\delta(\ve_1,\ve_2)}\omega_{i,j}=-\frac{\ve^i_1\ve^j_2}{nm}\cdot
2\pi \sqrt{-1}F\left(1-\frac{i}{n},1-\frac{j}{m},1;1-t\right).
\]
\end{lem}
\begin{proof}
Since 
\[
\int_{\delta(\ve_1,\ve_2)}\omega_{i,j}
=\int_{\delta(1,1)}\sigma(\ve_1,\ve_2)\omega_{i,j}
=\ve^i_1\ve^j_2\int_{\delta(1,1)}\omega_{i,j}
\]
we only need to show the case $\delta(1,1)$. Write $\zeta_m:=e^{2\pi \sqrt{-1}/m}$
and $\zeta_{2m}:=e^{\pi \sqrt{-1}/m}$
\begin{align*}
\int_{\delta(1,1)}\omega_{i,j}&=
(1-\zeta_m^j)\int_u\omega_{i,j}\\
&=-(1-\zeta_m^j)\int_ux^{i-1}y^{j-1}\frac{n^{-1}dy}{x^{n-1}(y^m-1)}\\
&=(1-\zeta_m^j)\int_uy^{j-1}
\left(\frac{t-y^m}{1-y^m}\right)^{\frac{i}{n}-1}\frac{n^{-1}dy}{1-y^m}\\
&=\frac{1-\zeta_m^j}{n}\int_uy^{j-1}(1-y^m)^{-\frac{i}{n}}(t-y^m)^{\frac{i}{n}-1}dy\\
&=\frac{\zeta_{2m}^{-j}-\zeta_{2m}^j}{n}\int_0^\infty
y^{j-1}(1+y^m)^{-\frac{i}{n}}(t+y^m)^{\frac{i}{n}-1}dy\\
\end{align*}
\begin{align*}
&=\frac{\zeta_{2m}^{-j}-\zeta_{2m}^j}{nm}\int_0^\infty
y^{\frac{j}{m}-1}(1+y)^{-\frac{i}{n}}(t+y)^{\frac{i}{n}-1}dy\\
&=\frac{\zeta_{2m}^{-j}-\zeta_{2m}^j}{nm}\int_1^\infty
(y-1)^{\frac{j}{m}-1}y^{-\frac{i}{n}}(t-1+y)^{\frac{i}{n}-1}dy\\
&=\frac{\zeta_{2m}^{-j}-\zeta_{2m}^j}{nm}\int_0^1
(y^{-1}-1)^{\frac{j}{m}-1}y^{\frac{i}{n}-2}(t-1+y^{-1})^{\frac{i}{n}-1}dy\\
&=\frac{\zeta_{2m}^{-j}-\zeta_{2m}^j}{nm}\int_0^1
(1-y)^{\frac{j}{m}-1}y^{-\frac{j}{m}}(1-(1-t)y)^{\frac{i}{n}-1}dy\\
&=\frac{\zeta_{2m}^{-j}-\zeta_{2m}^j}{nm}B\left(\frac{j}{m},1-\frac{j}{m}\right)
F\left(1-\frac{i}{n},1-\frac{j}{m},1,1-t\right)\\
&=\frac{\zeta_{2m}^{-j}-\zeta_{2m}^j}{nm}B\left(\frac{j}{m},1-\frac{j}{m}\right)
F\left(1-\frac{i}{n},1-\frac{j}{m},1,1-t\right)\\
&=-\frac{2\pi\sqrt{-1}}{nm}
F\left(1-\frac{i}{n},1-\frac{j}{m},1,1-t\right).
\end{align*}
\end{proof}

\begin{lem}\label{Fermat-lem3}
Let $T_1$ be the local monodromy at $t=1$.
There is a unique homology cycle $\gamma(\ve_1,\ve_2)\in H_1(X_t,\Q)$
such that $(T_1-1)\gamma(\ve_1,\ve_2)=\delta(\ve_1,\ve_2)$ and
\[
\int_{\gamma(\ve_1,\ve_2)}\omega_{i,j}=
\frac{\ve^i_1\ve^j_2}{nm}
B\left(1-\frac{i}{n},1-\frac{j}{m}\right)
F\left(1-\frac{i}{n},1-\frac{j}{m},2-\frac{i}{n}-\frac{j}{m};t\right).
\]
\end{lem}
\begin{proof}
The uniqueness follows from the fact that 
the monodromy invariant part of $H_1(X_t)$ is trivial.
We show the existence.
Write $\Ev:=\langle\delta(\ve_1,\ve_2)\mid(\ve_1,\ve_2)\in \mu_n\times\mu_m
\rangle\subset H_1(X_t,\Q)$.
Then it follows from the last condition of HG fibration in Definition \ref{HG-defn}
that one has
\[
N_1:=T_1-1:H_1(X_t,\Q)/\Ev\os{\cong}{\lra}\Ev.
\]
Therefore there is a unique homology cycle $\gamma(\ve_1,\ve_2)\in H_1(X_t,\Q)$
such that $(T_1-1)\gamma(\ve_1,\ve_2)=\delta(\ve_1,\ve_2)$ up to $\Ev$.
Let $T_0$ be the local monodromy at $t=0$.
Since $T_0-1:\Ev\to \Ev$ is bijective, we can choose 
$\gamma(\ve_1,\ve_2)$ such that $(T_0-1)\gamma(\ve_1,\ve_2)=0$
by replacing $\gamma(\ve_1,\ve_2)$ with $\gamma(\ve_1,\ve_2)+\delta_0$.
Then we show that this gives the desired cycle.
The monodromy of Gauss hypergeometric functions is well-known, in particular,
\[
(T_1-1)B(a,b)F(a,b,a+b;t)=-2\pi\sqrt{-1}F(a,b,1;1-t).
\]
Therefore letting 
\[
f_1:=-2\pi \sqrt{-1}F\left(1-\frac{i}{n},1-\frac{j}{m},1;1-t\right)
\]
\[
f_2:=B\left(1-\frac{i}{n},1-\frac{j}{m}\right)
F\left(1-\frac{i}{n},1-\frac{j}{m},2-\frac{i}{n}-\frac{j}{m};t\right)
\]
we have
\[
(T_1-1)\int_{\gamma(\ve_1,\ve_2)}\omega_{i,j}=
\int_{\delta(\ve_1,\ve_2)}\omega_{i,j}
=\frac{\ve^i_1\ve^j_2}{nm}f_1=(T_1-1)\frac{\ve^i_1\ve^j_2}{nm}f_2.
\]
On the other hand since $(T_0-1)\gamma(\ve_1,\ve_2)=0$,
we have 
\[
(T_0-1)\int_{\gamma(\ve_1,\ve_2)}\omega_{i,j}=0=(T_0-1)\frac{\ve^i_1\ve^j_2}{nm}f_2.
\]
Thus
\[
F:=\int_{\gamma(\ve_1,\ve_2)}\omega_{i,j}-\frac{\ve^i_1\ve^j_2}{nm}f_2
\]
is invariant under the both local monodromy, and this means $F=0$.
\end{proof}

\subsection{HG fibration of Gauss type}\label{Gauss-sect}
Suppose that the characteristic of $k$ is $0$. 
Let $f:X\to \P^1$ be the fibration
whose general fiber is the smooth completion of 
an affine curve 
\[y^N=x^a(1-x)^b(1-tx)^{N-b},\quad 0<a,b<N,\, \gcd(N,a,b)=1.\] 
$f$ is smooth over $\P^1\setminus\{0,1,\infty\}$.
Suppose that $k^\times$ contains all $N$-th roots of unity, and denote by
$\mu_N\subset k^\times$ the group of all $N$-th roots.
The action $(x,y,t)\mapsto (x,\zeta_Ny,t)$ for $\zeta_N\in \mu_N$ gives
a multiplication by the group ring $R=\Q[\mu_N]$.
Then $f$ is a HG fibration with multiplication by $(R,e)$
if and only if a projection $e:\Q[\mu_N]\to E$ satisfies $ad/N\not\in\Z$ and $bd/N\not\in\Z$ where 
$d:=\sharp\ker[e:\mu_N\to E^\times]$ (\cite{a-o-2} \S 3.2).
\begin{lem}\label{Gauss-lem1}
Let $X_t=f^{-1}(t)$ denotes the general fiber.
Put a $1$-form
\[
\omega_n:=\frac{x^{p_n}(1-x)^{q_n}(1-tx)^{n-1-q_n}}{y^n}dx,
\quad p_n:=\lfloor\frac{an}{N}\rfloor,\quad q_n:=\lfloor\frac{bn}{N}\rfloor
\]
for $n\in \{1,2,\ldots,N-1\}$.
Put $d:=\sharp\ker[e:\mu_N\to E^\times]$ and
\[
I_e:=\{n\in\Z\mid 1\leq n\leq N-1,\, d|n,\, \gcd(n/d,N/d)=1\}.
\]
Then $\{\omega_n\mid n\in I_e\}$ forms a basis of 
the $e$-part $\vg(X_t,\Omega^1_{X_t})(e)$.
\end{lem}
\begin{proof}
\cite{archinard} (13), p.917.
\end{proof}

\begin{lem}\label{Gauss-lem3}
Suppose $k=\C$. 
Write $a_n:=\{an/N\}$ and $b_n:=\{bn/N\}$.
There are points $P_0,P_1\in X_t$ such that $x=0,1$
 and a homology cycle
\[
u_0\in H_1^B(X_t,\{P_0,P_1\};\Z)
\]
such that
\[
\int_{u_0}\omega_n=B(a_n,b_n)F(a_n,b_n,a_n+b_n;t)\quad \mbox{for }|t|\ll1.
\]
Moreover letting $T_1$ be the local monodromy at $t=1$ and 
$u_1:=(1-T_1)u_0$, we have
\[
\int_{u_1}\omega_n=2\pi\sqrt{-1}F(a_n,b_n,1;1-t).
\]
The $e$-part $H_1^B(X_t,\Q)(e)$ is spanned by
\[
\gamma_0:=(1-\sigma)u_0,\quad
\gamma_1:=(1-\sigma)u_1
\]
as $E$-module where $\sigma$ is an automorphism of $X_t$
given by $(x,y)\mapsto(x,e^{\frac{2\pi\sqrt{-1}}{N}} y)$.
\end{lem}
\begin{proof}
Define a path $u_0$ as
\[
(x,y)=(s,s^\frac{a}{N}(1-s)^\frac{b}{N}(1-ts)^{1-\frac{b}{N}}),\quad s\in [0,1]
\]
in which $s^\frac{a}{N}$, $(1-s)^\frac{b}{N}$ take values in $\R_{\geq0}$
and $(1-ts)^{1-\frac{b}{N}}$ takes values such that 
$|(1-ts)^{1-\frac{b}{N}}-1|\ll1$.
Then
\begin{align*}
\int_{u_0}\omega_n&=\int_0^1x^{a_n-1}(1-x)^{b_n-1}(1-tx)^{-b_n}dx\\
&=B(a_n,b_n)F(a_n,b_n,a_n+b_n;t).
\end{align*}
The assertion for $u_1$ follows from the fact 
\[
(T_1-1)B(a,b)F(a,b,a+b;t)=-2\pi\sqrt{-1}F(a,b,1;1-t).
\]
The last assertion follows from the fact that $\dim_EH^B_1(X_t,\Q)(e)=2$
and that $\gamma_0$ and $\gamma_1$ are $E$-linearly independent
because their images by the map 
\[
H^B_1(X_t,\Q)(e)\lra \O\omega_n^\vee=\Hom(\O\omega_n,\O),\quad \gamma\longmapsto \int_\gamma\omega_n
\]
are $\C$-linearly independent.
\end{proof}
\begin{lem}\label{Gauss-lem2}
Let the notation be as in Lemma \ref{Gauss-lem1}. Then
\[
\vg(X,\Omega^2_X(\log Y))(e)=\bigoplus_{n\in I_e}k\cdot
\frac{dt}{t-1}\omega_n.
\]
\end{lem}
\begin{proof}
Let $\chi:R\to k$ be a homomorphism of $\Q$-algebra factoring through $e$.
Write
\[
\vg(X,\Omega^2_X(\log Y))(\chi):=k\ot_{\chi,k\ot_\Q R}\vg(X,\Omega^2_X(\log Y)).
\]
Then the assertion is equivalent to that for any $\chi$
\begin{equation}\label{Gauss-lem2-eq1}
\vg(X,\Omega^2_X(\log Y))(\chi)=k\cdot
\frac{dt}{t-1}\omega_n,
\end{equation}
where $n\in \{1,\ldots,N-1\}$ such that $\chi(\zeta)=\zeta^{-n}$ 
for $\forall \zeta\in\mu_N$.

We may suppose $k=\C$.
Put $Y_0=f^{-1}(0)$, $Y_\infty=f^{-1}(\infty)$, $S=\P^1\setminus\{0,1,\infty\}$ and 
$U:=X\setminus(Y\cup Y_0\cup Y_\infty)=f^{-1}(S)$.
Let $\cH=H^1_\dR(U/S)$ be a connection.
Then 
\begin{align*}
\vg(X,\Omega^2_X(\log Y+Y_0+Y_\infty))&=
F^2H^2_\dR(U)\\
&=F^2H^1_\dR(S,\cH)\\
&=\vg(\P^1,\Omega^1_{\P^1}(\log(0+1+\infty))\ot \cH_e)
\end{align*}
where $\cH_e\subset j_*\cH$, $j:S\hra\P^1$ is Deligne's canonical extension.
Hence
\[
\vg(X,\Omega^2_X(\log Y+Y_0+Y_\infty))(\chi)=
\vg(\P^1,\Omega^1_{\P^1}(\log(0+1+\infty))\ot \cH_e(\chi)).
\]
Note that $X$ is a nonsingular rational surface (Lemma \ref{boundary-lem1} below).
The localization sequence induces an isomorphism
\[
\Res:F^2H^2_\dR(U)\os{\cong}{\lra} F^1H^\dR_1(Y)\op F^1H^\dR_1(Y_0)\op F^1H^\dR_1(Y_\infty)
\]
by the Poincare residue map.
Since $a_n,b_n\not\in \Z$, the local monodromy at $t=\infty$ on $H^1(X_t,\Q)$
has no eigenvalue $1$ by Lemma \ref{Gauss-lem3}. This implies
the composition $H^2_\dR(U)(e)\to H^\dR_1(Y_\infty)$ is zero. Hence
$H^\dR_1(Y_\infty)(e)=0$ and
\[
\vg(X,\Omega^2_X(\log Y+Y_0+Y_\infty))(\chi)=
\vg(X,\Omega^2_X(\log Y+Y_0))(\chi).
\]
Summing up the above we have
\[
0\to \vg(X,\Omega^2_X(\log Y))(\chi)\to
\vg(\P^1,\Omega^1_{\P^1}(\log(0+1+\infty))\ot \cH_e(\chi))
\os{\Res}{\to}H^\dR_1(Y_0) \to0.
\]
By a computation of the periods in Lemma \ref{Gauss-lem3}, one can get a
explicit description of $\cH_e$ and then 
\[
\vg(\P^1,\Omega^1_{\P^1}(\log(0+1+\infty))\ot \cH_e(\chi))=
\begin{cases}
\langle\frac{dt}{t}\omega_n,\,\frac{dt}{t-1}\omega_n\rangle_\C &a_n+b_n\leq 1\\
\langle\frac{dt}{t-1}\omega_n\rangle_\C &a_n+b_n> 1.
\end{cases}
\]
(the details are left to the reader because it is a tedious computation, but see the proof of \cite{a-o-log} Lemma 3.7). Now \eqref{Gauss-lem2-eq1} is immediate.
\end{proof}

\section{Regulators of $K_2$ of HG fibration of Fermat type}\label{m-fermat-sect}
In this section the base field is $\C$.

\subsection{}
Let $X$ be a smooth proper variety over $\C$.
Let
\begin{equation}\label{m-fermat-reg}
\reg:H^p_\cM(X,\Z(q))\lra H^p_\cD(X,\Z(q))
\end{equation}
be the {\it Beilinson regulator map}
from the motivic cohomology group
to the Deligne-Beilinson cohomology group (cf. \cite{schneider}).
If $p\leq q$ and $p\ne 2q$, then the right hand side is canonically isomorphic to
$\Hom(H^B_{p-1}(X,\Z),\C/\Z(q))$ modulo torsion.
For $\xi\in H^p_\cM(X,\Z(q))$ and $\gamma\in H^B_{p-1}(X,\Z)$, 
we write the pairing by
\[
\langle\reg(\xi)\,|\,\gamma\rangle\in\C/\Z(q).
\]
\begin{prop}\label{m-fermat-prop1}
Let $f:U\to S$ be a smooth proper morphism onto a smooth curve $S$ over $\C$.
Let $U_t=f^{-1}(t)$ denotes a fiber.
Suppose $p=q\geq 1$. 
Let $\xi\in H^p_\cM(U,\Z(p))$ and $\gamma_t\in H_{p-1}^B(U_t,\Z)$.
We think of
\[F=\langle\reg(\xi|_{U_t})\,|\,\gamma_t\rangle\]
being a multi-valued function of variable $t$ which is locally holomorphic on 
$t\in S$.
Let
\[
\mathrm{dlog}(\xi)=dt\wedge\omega\in \vg(U,\Omega^p_U).
\]
Then
\[
\frac{dF}{dt}=\pm
\int_{\gamma_t}\omega.
\]
If $p=q=2$, let $\xi=\sum\{f,g\}$ be a $K_2$-symbol, and
describe
\[
F=\sum\int_{\gamma_t}
\log f\frac{dg}{g}-\log g(O)\frac{df}{f}
\]
by Beilinson's formula where $O$ is the origin of a loop $\gamma_t$
(e.g.\,\cite{hain} Proposition 6.3).
Then
\[
\frac{dF}{dt}=\int_{\gamma_t}\omega,\quad 
\mbox{where }\sum\frac{df}{f}\frac{dg}{g}=dt\wedge \omega.
\]
\end{prop}
\begin{proof}
The regulator map \eqref{m-fermat-reg} sits into a commutative
diagram
\[
\xymatrix{
H^p_\cM(U,\Q(p))\ar[r]^{\reg_S\quad}\ar[d]&\Ext_S(\Q,R^{p-1}f_*\Q(p))\ar[d]\\
H^p_\cM(U_t,\Q(p))\ar[r]^{\reg\qquad}&\Ext(\Q,H^{p-1}(U_t,\Q(p)))
}
\]
where $\Ext_S$ (resp. $\Ext$) denotes the group of
1-extensions of admissible variations of MHS's
(resp. MHS's), and the vertical arrows are the restriction maps.
Let 
\[
0\lra R^{p-1}f_*\Q(p)\lra {\mathscr V}\lra \Q\lra 0
\]
be the corresponding 1-extension to $\reg_S(\xi)$.
Let $e_\dR\in {\mathscr V}_\dR\cap F^0$ and
$e_B\in {\mathscr V}_B$ be local liftings of $1\in\Q$.
Then $e_\dR-e_B\in R^{p-1}f_*\Q(p)$, and 
\[
F=\pm\langle e_\dR-e_B,\gamma_t\rangle
\]
where $\langle-,-\rangle:H^{p-1}(X_t,\C)\ot H_{p-1}^B(X_t,\C)\to\C$ is the
natural pairing. Fix a lifting $\wt{\gamma}_t\in {\mathscr V}^\vee_B\ot\Q(p)$
via the surjective map ${\mathscr V}^\vee_B\ot\Q(p)\to H_{p-1}^B(X_t,\Q)$.
Then one has
\[
\pm F=\langle e_\dR,\wt\gamma_t\rangle-
\overbrace{\langle e_B,\wt\gamma_t\rangle}^{\Q(p)}
\]
and hence
\[
\pm\frac{dF}{dt}=\frac{d}{dt}\langle e_\dR,\wt\gamma_t\rangle
=\langle \nabla(e_\dR),\gamma_t\rangle
\]
where the last pairing is the natural pairing
on $\Omega^1_S\ot H^{p-1}_\dR(U/S)$ and $H^B_{p-1}(X_t,\C)$.
Note that $\nabla(e_\dR)$ is the extension data of
\[
0\lra H^{p-1}_\dR(U/S)\lra {\mathscr V}_\dR\lra \O_S\lra 0.
\]
and this corresponds to 
the de Rham realization of $\xi$. Hence $\nabla(e_\dR)=\dlog(\xi)$, and
the former assertion follows. The latter assertion follows from this and the fact that
$d(\int_\gamma \eta)=(\int_\gamma\omega)dt$ for 1-forms $\eta$ and $\omega$ such that $d\eta=dt\wedge \omega$.
\end{proof}

\subsection{Main Theorems}\label{m-fermat-sect2}
Let $f$ be a HG fibration of Fermat type,
\[
X_t=f^{-1}(t):(x^n-1)(y^m-1)=1-t,\quad n,\,m\geq 2
\]
on which the group $\mu_n\times \mu_m$ acts where $\mu_n\subset \C^\times$
denotes the group of $n$-th roots of unity.
We then discuss the Beilinson regulator map
\[
\reg:H^2_\cM(X_t,\Q(2))=K_2(X_t)^{(2)}\lra H^2_\cD(X_t,\Q(2))=\Hom(H^B_1(X_t,\Z),\C/\Q(2)).
\]
For $(\nu_1,\nu_2)\in\mu_n\times \mu_m$ such that $\nu_1,\nu_2\ne1$,
we consider a $K_2$-symbol
\begin{equation}\label{m-fermat-eq1}
\xi=\left\{
\frac{x-1}{x-\nu_1},\frac{y-1}{y-\nu_2}
\right\}\in K_2(X\setminus f^{-1}(1)).
\end{equation}
One immediately has
\begin{equation}\label{m-fermat-eq2}
\mathrm{dlog}(\xi)=-
\sum_{i=1}^{n-1}\sum_{j=1}^{m-1}(1-\nu^{-i}_1)(1-\nu^{-j}_2)\frac{dt}{t-1}
\omega_{i,j}.
\end{equation}
The main theorems are formulas describing 
\[
\langle\reg(\xi)\,|\,\gamma\rangle=\langle\reg(\xi|_{X_t})\,|\,\gamma\rangle\in\C/\Q(2),
\quad \gamma\in H_1^B(X_t,\Q)
\]
via the generalized hypergeometric functions.
\begin{thm}\label{m-fermat-thm1}
Write $a_i:=1-i/n$ and $b_j:=1-j/m$. 
Let $\delta(\ve_1,\ve_2)$ be the homology cycle as in \S \ref{Fermat-sect}.
Then for $|t-1|<1$
\begin{align*}
\frac{1}{2\pi\sqrt{-1}}\langle\reg(\xi)\mid\delta(\ve_1,\ve_2)\rangle
=C_0+&C_1\log(1-t)
+\sum_{i=1}^{n-1}\sum_{j=1}^{m-1}
(1-\nu_1^{-i})(1-\nu_2^{-j})\frac{\ve_1^i\ve_2^j}{nm}\\
&\times
a_ib_j(1-t)\,{}_4F_3\left({a_i+1,b_j+1,1,1\atop 2,2,2};1-t\right)
\end{align*}
modulo $\Q(1)=2\pi\sqrt{-1}\Q$
where
\[
C_0=\begin{cases}
-\log(nm(1-\nu_1)(1-\nu_2))&(\ve_1,\ve_2)=(1,1),(\nu_1,\nu_2)\\
\log(nm(1-\nu_1)(1-\nu_2))&(\ve_1,\ve_2)=(1,\nu_2),(\nu_1,1)\\
\log\left(\frac{\ve_2-1}{\ve_2-\nu_2}\right)&\ve_1=1\mbox{ and }\ve_2\ne1,\nu\\
\log\left(\frac{\ve_1-1}{\ve_1-\nu_1}\right)&\ve_1\ne1,\nu_1\mbox{ and }\ve_2=1\\
-\log\left(\frac{\ve_2-1}{\ve_2-\nu_2}\right)&\ve_1=\nu_1\mbox{ and }\ve_2\ne1,\nu\\
-\log\left(\frac{\ve_1-1}{\ve_1-\nu_1}\right)&\ve_1\ne1,\nu_1\mbox{ and }\ve_2=\nu_2\\
0&\mbox{others}
\end{cases}
\]
\[
C_1=\begin{cases}
1&(\ve_1,\ve_2)=(1,1),(\nu_1,\nu_2)\\
-1&(\ve_1,\ve_2)=(1,\nu_2),(\nu_1,1)\\
0&\mbox{others}.
\end{cases}
\]
\end{thm}
\begin{rem}\label{m-fermat-rem1}
It is worth noting 
\[
C_0=-\sum_{i=1}^{n-1}\sum_{j=1}^{m-1}
(1-\nu_1^{-i})(1-\nu_2^{-j})\frac{\ve_1^i\ve_2^j}{nm}
(2\psi(1)-\psi(a_i)-\psi(b_j))\mod \Q(1)
\]
\[
C_1=\sum_{i=1}^{n-1}\sum_{j=1}^{m-1}
(1-\nu_1^{-i})(1-\nu_2^{-j})\frac{\ve_1^i\ve_2^j}{nm}
\]
where $\psi(x)=\Gamma'(x)/\Gamma(x)$ is the digamma function.
Hence we can rewrite
\begin{align*}
&\frac{1}{2\pi\sqrt{-1}}\langle\reg(\xi)\mid\delta(\ve_1,\ve_2)\rangle
=\sum_{i=1}^{n-1}\sum_{j=1}^{m-1}
(1-\nu_1^{-i})(1-\nu_2^{-j})\frac{\ve_1^i\ve_2^j}{nm}\\
&\times\left(-2\psi(1)+\psi(a_i)+\psi(b_j)+\log(1-t)
+a_ib_j(1-t)\,{}_4F_3\left({a_i+1,b_j+1,1,1\atop 2,2,2};1-t\right)\right)
\end{align*}
modulo $\Q(1)$.
\end{rem}

\medskip

\noindent{\it Proof of Theorem \ref{m-fermat-thm1}.}
Put
\[
F:=\frac{1}{2\pi\sqrt{-1}}\langle\reg(\xi)\mid\delta(\ve_1,\ve_2)\rangle.
\]
By Proposition \ref{m-fermat-prop1} and \eqref{m-fermat-eq2}
\[
(t-1)\frac{dF}{dt}=-\sum_{i=1}^{n-1}\sum_{j=1}^{m-1}(1-\nu^{-i}_1)(1-\nu^{-j}_2)
\frac{1}{2\pi\sqrt{-1}}\int_{\delta(\ve_1,\ve_2)}\omega_{i,j}.
\]
By Lemma \ref{Fermat-lem2}
\begin{equation}\label{m-fermat-thm1-eq3}
(t-1)\frac{dF}{dt}=\sum_{i=1}^{n-1}\sum_{j=1}^{m-1}(1-\nu^{-i}_1)(1-\nu^{-j}_2)
\frac{\ve_1^i\ve_2^j}{nm}
F(a_i,b_j,1;1-t).
\end{equation}
This immediately implies the desired formula except the constant term ``$C_0$''.
Therefore it is enough to show 
\begin{equation}\label{m-fermat-thm1-eq1}
F=C_0+C_1\log(1-t)+o(t-1)\mod\Q(1)
\end{equation}
for $|t-1|\ll1$. Here ``$o(t-1)$'' denotes a continuous function which converges
to $0$ as $t\to 1$.
To do this we use Beilinson's formula (e.g.\,\cite{hain} Proposition 6.3)
\[
\langle\reg\{f,g\}\mid\gamma\rangle=\int_\gamma
\log f\frac{dg}{g}-\log g(O)\frac{df}{f}
\]
where $O$ is the origin of a loop $\gamma\in \pi_1(X_t,O)$
(it is important to fix the origin in the above formula).
We show \eqref{m-fermat-thm1-eq1} only in case $(\ve_1,\ve_2)=(1,1)$ (the others
are proven in a similar way):
\begin{equation}\label{m-fermat-thm1-eq2}
F=-\log(nm(1-\nu_1)(1-\nu_2))+\log(1-t)+o(1-t).
\end{equation}
Recall the loop $\delta:=\delta(1,1)$ with the origin $y=0$ from \eqref{Fermat-eq1}.
Beilinson's formula yields
\[
F=\frac{1}{2\pi\sqrt{-1}}
\int_\delta\left(\log\frac{x-1}{x-\nu_1}\dlog\frac{y-1}{y-\nu_2}-\log(\nu_2^{-1})
\dlog\frac{x-1}{x-\nu_1}\right).
\]
When $|t-1|\ll1$, $\delta$ is a circle in a neighborhood of $(x,y)=(1,1)$ and
\[
x-1=\frac{1-t}{y^m-1}(1+x+\cdots+x^{n-1})^{-1}
=\frac{1-t}{nm(y-1)}+o(t-1)
\] 
on $\delta$.
Therefore
\begin{align*}
\mbox{1st term of }F
&=\frac{1}{2\pi\sqrt{-1}}\int_\delta\log\left(
\frac{(1-t)(1-\nu_1)^{-1}}{nm(y-1)}\right)\dlog\frac{y-1}{y-\nu_2}+o(t-1)\\
&=\frac{1}{2\pi\sqrt{-1}}\oint_{y=1}\log\left(
\frac{(1-t)(1-\nu_1)^{-1}}{nm(y-1)}\right)\frac{dy}{y-1}\\
&\quad-\frac{1}{2\pi\sqrt{-1}}\oint_{y=1}\log\left(
\frac{(1-t)(1-\nu_1)^{-1}}{nm(y-1)}\right)\frac{dy}{y-\nu_2}+o(t-1)\\
&=\frac{-1}{2\pi\sqrt{-1}}\oint_{y=1}\log(y-1)\frac{dy}{y-1}
+\log \frac{(1-t)}{nm(1-\nu_1)}\\
&\quad+\frac{1}{2\pi\sqrt{-1}}\oint_{y=1}\log(y-1)\frac{dy}{y-\nu_2}+o(t-1)\\
&=\log \frac{(1-t)}{nm(1-\nu_1)}+\frac{1}{2\pi\sqrt{-1}}\oint_{y=1}\log(y-1)\frac{dy}{y-\nu_2}+o(t-1)\mod\Q(1)
\end{align*}
where ``$\oint_{y=1}$'' denotes the integral along a path with the origin
$y=0$ which goes around $y=1$ in the counter-clockwise diretion
and comes back to the origin (see figure).

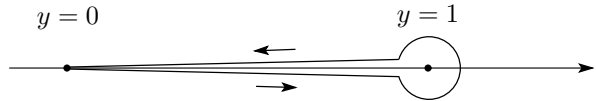
\begin{center}
%\input{K2.pic2.tex}
%WinTpicVersion4.30c
{\unitlength 0.1in%
\begin{picture}( 30.5000,  5.1200)(  3.6000,-17.7400)%
% VECTOR 2 0 3 0 Black White
% 2 360 1610 3410 1610
% 
\special{pn 8}%
\special{pa 360 1610}%
\special{pa 3410 1610}%
\special{fp}%
\special{sh 1}%
\special{pa 3410 1610}%
\special{pa 3343 1590}%
\special{pa 3357 1610}%
\special{pa 3343 1630}%
\special{pa 3410 1610}%
\special{fp}%
% DOT 0 0 3 0 Black White
% 1 2550 1610
% 
\special{pn 4}%
\special{sh 1}%
\special{ar 2550 1610 16 16 0  6.28318530717959E+0000}%
% DOT 0 0 3 0 Black White
% 1 660 1610
% 
\special{pn 4}%
\special{sh 1}%
\special{ar 660 1610 16 16 0  6.28318530717959E+0000}%
% CIRCLE 2 0 3 0 Black White
% 4 2555 1610 2425 1710 2395 1655 2390 1560
% 
\special{pn 8}%
\special{ar 2555 1610 164 164  3.4358272  2.8674252}%
% LINE 2 0 3 0 Black White
% 4 655 1605 2395 1565 655 1615 2395 1655
% 
\special{pn 8}%
\special{pa 655 1605}%
\special{pa 2395 1565}%
\special{fp}%
\special{pa 655 1615}%
\special{pa 2395 1655}%
\special{fp}%
% VECTOR 2 0 3 0 Black White
% 4 1854 1512 1644 1520 1652 1706 1856 1712
% 
\special{pn 8}%
\special{pa 1854 1512}%
\special{pa 1644 1520}%
\special{fp}%
\special{sh 1}%
\special{pa 1644 1520}%
\special{pa 1711 1537}%
\special{pa 1697 1518}%
\special{pa 1710 1497}%
\special{pa 1644 1520}%
\special{fp}%
\special{pa 1652 1706}%
\special{pa 1856 1712}%
\special{fp}%
\special{sh 1}%
\special{pa 1856 1712}%
\special{pa 1790 1690}%
\special{pa 1803 1710}%
\special{pa 1789 1730}%
\special{pa 1856 1712}%
\special{fp}%
% STR 2 0 3 0 Black White
% 4 665 1295 665 1345 5 0 0 0
% $y=0$
\put(6.6500,-13.4500){\makebox(0,0){$y=0$}}%
% STR 2 0 3 0 Black White
% 4 2550 1285 2550 1335 5 0 0 0
% $y=1$
\put(25.5000,-13.3500){\makebox(0,0){$y=1$}}%
\end{picture}}%

\end{center}

Since
\begin{align*}
\frac{1}{2\pi\sqrt{-1}}\int_\delta\log(y-1)\frac{dy}{y-\nu_2}
&=
\log(-\nu_2)-\frac{1}{2\pi\sqrt{-1}}\oint_{y=1}
\log(y-\nu_2)\frac{dy}{y-1}\\
&=\log(-\nu_2)-\log(1-\nu_2)\\
&=\log(\nu_2)-\log(1-\nu_2)\mod\Q(1)
\end{align*}
we have
\[
\mbox{1st term of }F=
\log\left( \frac{1-t}{nm(1-\nu_1)(1-\nu_2)}\right)+\log(\nu_2)+o(t-1)\mod\Q(1).
\]
On the other hand
\begin{align*}
\mbox{2nd term of }F&=
\frac{-1}{2\pi\sqrt{-1}}\int_\delta\log(\nu_2^{-1})
\dlog\left(\frac{1-t}{nm(y-1)}\right)+o(t-1)\\
&=
\log(\nu_2^{-1})+o(t-1).
\end{align*}
We thus have
\[F=\log\left( \frac{1-t}{nm(1-\nu_1)(1-\nu_2)}\right)+o(t-1).\]
the desired assertion \eqref{m-fermat-thm1-eq2}.
This completes the proof of Theorem \ref{m-fermat-thm1}. \qed

\medskip

\begin{thm}\label{m-fermat-thm3}
Put $z:=(1-t)^{-1}$ and
\[
C_{a,b}
:=\frac{\sin(\pi a)}{\pi}B_{a,b}
=
\frac{\Gamma(b-a)}{\Gamma(1-a)\Gamma(b)}.
\]
Then we have an alternative description of the regulators in 
Theorem \ref{m-fermat-thm1}
\begin{multline*}
\frac{1}{2\pi\sqrt{-1}}\langle\reg(\xi)\mid\delta(\ve_1,\ve_2)\rangle
=-\sum_{i=1}^{n-1}\sum_{j=1}^{m-1}
(1-\nu_1^{-i})(1-\nu_2^{-j})\frac{\ve_1^i\ve_2^j}{nm}\\
\times(a_i^{-1}C_{a_i,b_j}(-z)^{a_i}\,F_{a_i,b_j}(z)
+b_j^{-1}C_{b_j,a_i}(-z)^{b_j}\,F_{b_j,a_i}(z))
\end{multline*}
modulo $\Q(1)=2\pi\sqrt{-1}\Q$.
\end{thm}
\begin{proof}
This is immediate from Remark \ref{m-fermat-rem1} and the following lemma due to W. Zudilin.
\end{proof}
\begin{lem}[Zudilin]\label{m-fermat-thm4}
Let $z=(1-t)^{-1}$. Then
\begin{align*}
&\pi i+2\psi(1)-\psi(a)-\psi(b)-\log(1-t)
-ab(1-t)\,{}_4F_3\left({a+1,b+1,1,1\atop 2,2,2};1-t\right)\\
&=a^{-1}C_{a,b}(-z)^a\,{}_3F_2\left({a,a,a\atop 1+a-b,a+1};z\right)
+b^{-1}C_{b,a}(-z)^b\,{}_3F_2\left({b,b,b\atop 1-a+b,b+1};z\right).
\end{align*}
\end{lem}
\begin{proof}
Apply $z\frac{d}{dz}$ on the both side. 
Then it turns out,
\[%\begin{equation}\label{branch-eq3}
{}_2F_1(a,b,1;1-t)=C_{a,b}(-z)^aF(a,a,1+a-b;z)+C_{b,a}(-z)^bF(b,b,1-a+b;z),
\]%\end{equation}
and this is valid (\cite{NIST} 15.8.2).
This proves Lemma \ref{m-fermat-thm4} modulo constant.
To remove `modulo constant', 
 we consider the limit $t\to 0^-$ so that $z=(1-t)^{-1}\to 1^-$.

Recall the formula (4.3.3) in Slater's book \cite{slater} (also on page 15 of Bailey's book \cite{Bailey}), which we write for the case $d=e=1$ for the choice $+$ of the sign in the exponent:
\begin{align*}
&
\Gamma(a)\Gamma(b)\Gamma(c)\,{}_3F_2\biggl(\begin{matrix} a, \, b, \,c \\ 1, \, 1 \end{matrix}\biggm|1\biggr)
\\ &\quad
=e^{\pi ia}\frac{\Gamma(a)\Gamma(b-a)\Gamma(c-a)}{\Gamma(1-a)^2}\,{}_3F_2\biggl(\begin{matrix} a, \, a, \, a \\ 1+a-b, \, 1+a-c \end{matrix}\biggm|1\biggr)
\\ &\quad\qquad
+e^{\pi ib}\frac{\Gamma(b)\Gamma(a-b)\Gamma(c-b)}{\Gamma(1-b)^2}\,{}_3F_2\biggl(\begin{matrix} b, \, b, \, b \\ 1+b-a, \, 1+b-c \end{matrix}\biggm|1\biggr)
\\ &\quad\qquad
+e^{\pi ic}\frac{\Gamma(c)\Gamma(a-c)\Gamma(b-c)}{\Gamma(1-c)^2}\,{}_3F_2\biggl(\begin{matrix} c, \, c, \, c \\ 1+c-a, \, 1+c-b \end{matrix}\biggm|1\biggr).
\end{align*}
Dividing the both sides by $\Gamma(a)\Gamma(b)$, putting the two summands on the right on one side and taking the limit as $c\to0$ we get
\begin{align*}
&
a^{-1}e^{\pi ia}C_{a,b}\cdot{}_3F_2\biggl(\begin{matrix} a, \, a, \, a \\ 1+a-b, \, 1+a \end{matrix}\biggm|1\biggr)
+b^{-1}e^{\pi ia}C_{b,a}\cdot{}_3F_2\biggl(\begin{matrix} b, \, b, \, b \\ 1+b-a, \, 1+b \end{matrix}\biggm|1\biggr)
\\ &\quad
=\lim_{c\to0}\biggl(e^{\pi ic}\frac{\Gamma(c)\Gamma(a-c)\Gamma(b-c)}{\Gamma(a)\Gamma(b)\Gamma(1-c)^2}
\,{}_3F_2\biggl(\begin{matrix} c, \, c, \, c \\ 1+c-a, \, 1+c-b \end{matrix}\biggm|1\biggr)
\\ &\quad\qquad\qquad\qquad
-\Gamma(c)\,{}_3F_2\biggl(\begin{matrix} a, \, b, \,c \\ 1, \, 1 \end{matrix}\biggm|1\biggr)\biggr).
\end{align*}
We have
\begin{align*}
\lim_{c\to0}\Gamma(c)\biggl({}_3F_2\biggl(\begin{matrix} a, \, b, \,c \\ 1, \, 1 \end{matrix}\biggm|1\biggr)-1\biggr)
&=\lim_{c\to0}\sum_{n=1}^\infty\frac{(a)_n(b)_n\Gamma(c+n)}{n!^3}
=\sum_{n=1}^\infty\frac{(a)_n(b)_n\Gamma(n)}{n!^3}
\\
&=\sum_{n=0}^\infty\frac{(a)_{n+1}(b)_{n+1}n!}{(n+1)!^3}
=ab\,{}_4F_3\biggl(\begin{matrix} 1+a, \, 1+b, \, 1, \, 1 \\ 2, \, 2, \, 2 \end{matrix}\biggm|1\biggr).
\end{align*}
We are left with the limit
$$
L=\lim_{c\to0}\Gamma(c)\biggl(\frac{e^{\pi ic}\Gamma(a-c)\Gamma(b-c)}{\Gamma(a)\Gamma(b)\Gamma(1-c)^2}
\,{}_3F_2\biggl(\begin{matrix} c, \, c, \, c \\ 1+c-a, \, 1+c-b \end{matrix}\biggm|1\biggr)-1\biggr).
$$
We now use $\Gamma(c)=\Gamma(1+c)/c\sim1/c$, $e^{\pi ic}=1+\pi ic+o(c)$,
$\Gamma(x-c)=\Gamma(x)-\Gamma'(x)c+o(c)=\Gamma(x)(1-\psi(x)c)+o(c)$ for $x\in\{a,b,1\}$, and
\begin{align*}
\frac{(c)_n^3}{n!\,(1+c-a)_n(1+c-b)_n}
&=c^3\cdot\frac{(n-1)!}{n!\,(1-a)_n(1-b)_n}+o(c^3)
\\
&=o(c^2) \quad\text{for}\; n=1,2,3,\dots,
\end{align*}
as $c\to0$. This implies that
\begin{align*}
L
&=\lim_{c\to0}\frac1c\biggl(\frac{e^{\pi ic}\Gamma(a-c)\Gamma(b-c)}{\Gamma(a)\Gamma(b)\Gamma(1-c)^2}-1\biggr)
\\
&=\lim_{c\to0}\frac1c\biggl(\frac{(1+\pi ic+o(c))(1-\psi(a)c+o(c))(1-\psi(b)c+o(c))}{(1-\psi(1)c+o(c))^2}-1\biggr)
\\
&=\pi i-\psi(a)-\psi(b)+2\psi(1)
\end{align*}
as required.
\end{proof}

\begin{thm}\label{m-fermat-thm2}
Let $e:\Q[\mu_n\times\mu_m]\to E$ be a projection onto a number field $E$
which does not factor through projections 
$\mu_n\times\mu_m\to\mu_n$ or $\mu_n\times\mu_m\to\mu_m$.
We denote by $\reg(\xi)(e)\in \Hom(H^B_1(X_t,\Q)(e),\C/\Q(2))$ the $e$-part,
and $\gamma(e)\in H^B_1(X_t)(e)$ as well. 
Let $I_e$ be the set of indices as in \eqref{Fermat-lem1-eq1}.
Put $z:=(1-t)^{-1}$ and
\[
B_{a,b}:=
B(a,b-a)=\frac{\Gamma(a)\Gamma(b-a)}{\Gamma(b)},
\quad F_{a,b}(z):={}_3F_2\left(
{a,a,a\atop 1+a-b,a+1};z
\right).
\]
Write $a_i:=1-i/n$ and $b_j:=1-j/m$. Assume 
\begin{equation}\label{m-fermat-thm2-eq1}
a_i\ne b_j\,(\Leftrightarrow\, i/n\ne j/m),\quad \forall\,(i,j)\in I_e
\end{equation}
or equivalently the diagram
\[
\xymatrix{
\hat\Z(1)=\varprojlim_l\mu_l\ar[r]^{\mathrm{can}}\ar[d]_{\mathrm{can}}
&\mu_n\cong\mu_n\times\{1\}\ar[d]^e\\
\mu_m\cong\{1\}\times\mu_m\ar[r]^e&E^\times
}
\]
does not commute. 
Then for $|t|<1$
\begin{align*}
\langle\reg(\xi)(e)\mid\gamma(\ve_1,\ve_2)(e)\rangle
=&-\sum_{(i,j)\in I_e}
(1-\nu_1^{-i})(1-\nu_2^{-j})\frac{\ve_1^i\ve_2^j}{nm}\\
&\times(a_i^{-1}B_{a_i,b_j}z^{a_i}\,F_{a_i,b_j}(z)
+b_j^{-1}B_{b_j,a_i}z^{b_j}\,F_{b_j,a_i}(z))
\end{align*}
modulo $\Q(2)$.
\end{thm}
\begin{proof}
Put
\[
F:=\langle\reg(\xi)(e)\mid\gamma(\ve_1,\ve_2)(e)\rangle.
\]
By Proposition \ref{m-fermat-prop1} and \eqref{m-fermat-eq2}
\[
(t-1)\frac{dF}{dt}=-\sum_{(i,j)\in I_e}(1-\nu^{-i}_1)(1-\nu^{-j}_2)
\int_{\gamma(\ve_1,\ve_2)}\omega_{i,j}.
\]
By Lemma \ref{Fermat-lem3}
\begin{equation}\label{branch-eq1}
(t-1)\frac{dF}{dt}=-\sum_{(i,j)\in I_e}(1-\nu^{-i}_1)(1-\nu^{-j}_2)
\frac{\ve_1^i\ve_2^j}{nm}
B(a_i,b_j)F(a_i,b_j,a_i+b_j;t).
\end{equation}
In particular $F$ is holomorphic at $|t|<1$.
For $a\ne b$
there is a formula (\cite{NIST} 15.8.3)
\begin{equation}\label{branch-eq2}
B(a,b)F(a,b,a+b;t)=B_{a,b}z^aF(a,a,1+a-b;z)+B_{b,a}z^bF(b,b,1-a+b;z).
\end{equation}
Here we take the branches such that
$F(a,b,a+b;t)$, $F(a,a,1+a-b;z)$ and $F(b,b,1-a+b;z)$ are holomorphic
on the region $|t|<1<|1-t|$ ($\Leftrightarrow$
$|z|<1$ and $\mathrm{Re}(z)>1/2$)
and $z^a$ and $z^b$ take the principal values on 
$|\mathrm{arg}(z)|< \pi$. 
Then \eqref{branch-eq1} and \eqref{branch-eq2}
immediately imply the desired formula except the constant term:
\[
F=C+\sum_{(i,j)\in I_e}
(1-\nu_1^{-i})(1-\nu_2^{-j})\frac{\ve_1^i\ve_2^j}{nm}(a_i^{-1}B_{a_i,b_j}z^{a_i}\,F_{a_i,b_j}(z)
+b_j^{-1}B_{b_j,a_i}z^{b_j}\,F_{b_j,a_i}(z)).
\]
To conclude $C\in\Q(2)$, we see the local monodromy $T_\infty$ at $t=\infty$.
We fix a loop $T_\infty$ such that the origin is a point in the interval $-1<t<0$
and it goes along the negative real axis and
turns around $t=\infty$.

\bigskip
%\input{K2.pic4.tex}

%WinTpicVersion4.30c
{\unitlength 0.1in%
\begin{picture}( 40.6000,  5.4600)(  4.0000, -9.8200)%
% VECTOR 2 0 3 0 Black White
% 2 400 690 4180 690
% 
\special{pn 8}%
\special{pa 400 690}%
\special{pa 4180 690}%
\special{fp}%
\special{sh 1}%
\special{pa 4180 690}%
\special{pa 4113 670}%
\special{pa 4127 690}%
\special{pa 4113 710}%
\special{pa 4180 690}%
\special{fp}%
% CIRCLE 0 0 3 0 Black White
% 4 890 700 1050 800 1050 800 1125 860
% 
\special{pn 20}%
\special{ar 890 700 189 189  0.5977584  0.5585993}%
% VECTOR 2 0 3 0 Black White
% 2 2410 770 2050 770
% 
\special{pn 8}%
\special{pa 2410 770}%
\special{pa 2050 770}%
\special{fp}%
\special{sh 1}%
\special{pa 2050 770}%
\special{pa 2117 790}%
\special{pa 2103 770}%
\special{pa 2117 750}%
\special{pa 2050 770}%
\special{fp}%
% VECTOR 2 0 3 0 Black White
% 2 2030 610 2420 610
% 
\special{pn 8}%
\special{pa 2030 610}%
\special{pa 2420 610}%
\special{fp}%
\special{sh 1}%
\special{pa 2420 610}%
\special{pa 2353 590}%
\special{pa 2367 610}%
\special{pa 2353 630}%
\special{pa 2420 610}%
\special{fp}%
% DOT 0 0 3 0 Black White
% 1 884 694
% 
\special{pn 4}%
\special{sh 1}%
\special{ar 884 694 16 16 0  6.28318530717959E+0000}%
% CIRCLE 2 0 3 0 Black White
% 4 886 720 958 472 944 472 814 468
% 
\special{pn 8}%
\special{ar 886 720 258 258  4.4340893  4.9421307}%
% VECTOR 2 0 3 0 Black White
% 2 938 470 952 474
% 
\special{pn 8}%
\special{pa 938 470}%
\special{pa 952 474}%
\special{fp}%
\special{sh 1}%
\special{pa 952 474}%
\special{pa 893 436}%
\special{pa 901 459}%
\special{pa 882 475}%
\special{pa 952 474}%
\special{fp}%
% CIRCLE 2 0 3 0 Black White
% 4 888 698 816 946 830 946 960 950
% 
\special{pn 8}%
\special{ar 888 698 258 258  1.2924967  1.8005381}%
% VECTOR 2 0 3 0 Black White
% 2 836 948 822 944
% 
\special{pn 8}%
\special{pa 836 948}%
\special{pa 822 944}%
\special{fp}%
\special{sh 1}%
\special{pa 822 944}%
\special{pa 881 982}%
\special{pa 873 959}%
\special{pa 892 943}%
\special{pa 822 944}%
\special{fp}%
% STR 2 0 3 0 Black White
% 4 895 570 895 620 5 0 0 0
% {\tiny$t=\infty$}
\put(8.9500,-6.2000){\makebox(0,0){{\tiny$t=\infty$}}}%
% STR 2 0 3 0 Black White
% 4 3976 542 3976 592 5 0 0 0
% {\tiny$t=0$}
\put(39.7600,-5.9200){\makebox(0,0){{\tiny$t=0$}}}%
% STR 2 0 3 0 Black White
% 4 4460 640 4460 690 1 0 0 0
% Figure of $T_\infty$
\put(44.6000,-6.9000){\makebox(0,0)[lt]{Figure of $T_\infty$}}%
% LINE 0 0 3 0 Black White
% 2 1080 690 3885 695
% 
\special{pn 20}%
\special{pa 1080 690}%
\special{pa 3885 695}%
\special{fp}%
% DOT 0 0 3 0 Black White
% 2 3974 690 3974 690
% 
\special{pn 4}%
\special{sh 1}%
\special{ar 3974 690 16 16 0  6.28318530717959E+0000}%
\special{sh 1}%
\special{ar 3974 690 16 16 0  6.28318530717959E+0000}%
% LINE 2 0 3 0 Black White
% 4 3866 722 3910 666 3906 724 3868 666
% 
\special{pn 8}%
\special{pa 3866 722}%
\special{pa 3910 666}%
\special{fp}%
\special{pa 3906 724}%
\special{pa 3868 666}%
\special{fp}%
\end{picture}}%

\bigskip

The eigenvalues of $T_\infty$ on $H^B_1(X_t)(e)$ are
those of the hypergeometric function $F(a_i,b_j,a_i+b_j;t)$, and hence they are
\[
p_i:=\exp(2\pi\sqrt{-1}a_i),\quad
q_j:=\exp(2\pi\sqrt{-1}b_j),\quad (i,j)\in I_e.
\]
Put a mondromy operator
\[
Q:=\prod_{(i,j)\in I_e}(T_\infty-p_i)(T_\infty-q_j)\in \Z[T_\infty].
\]
We use the notation in the proof of Proposition \ref{m-fermat-prop1}
\[
F=\langle e_\dR-e_B,\gamma_t\rangle
=\langle e_\dR,\wt\gamma_t\rangle-\langle e_B,\wt\gamma_t\rangle,\quad
\gamma_t:=\gamma(\ve_1,\ve_2).
\]
Note that $e_B=e_{B,t}$ may have nontrivial monodromy.
Apply $Q$ to the above, we then have
\[
QF=\overbrace{\langle e_\dR,Q\wt\gamma_t\rangle}^{\Q(2)}
-Q\overbrace{\langle e_{B,t},\wt\gamma_t\rangle}^{\Q(2)}\in\Q(2).
\]
On the other hand, since
$F_{a_i,b_j}(z)$ and $F_{b_j,a_i}(z)$ are holomorphic along the loop $T_\infty$ fixed above,
one has
\[
(T_\infty-p_i)(z^{a_i}F_{a_i,b_j}(z))=0,\quad
(T_\infty-q_j)(z^{b_j}F_{b_j,a_i}(z))=0.
\]
Therefore
\[
QF=QC=\left(\prod_{(i,j)\in I_e}(1-p_i)(1-q_j)\right)C\in \Q(2).\]
Since $p_i,q_j\ne1$ by definition, one concludes $C\in\Q(2)$. 
This completes the proof of Theorem \ref{m-fermat-thm1}.
\end{proof}
\section{Regulators of $K_2$ of HG fibration of Gauss type}\label{m-gauss-sect}
In this section the base field is $\C$.

\subsection{Construction of elements of $H^2_\cM(X_t,\Q(2))$}\label{K2-sect}
Let $f:X\to \P^1$ be a HG fibration defined with multiplication by $(R,e)$.
Let $Y=f^{-1}(1)_{\mathrm{red}}$ be the reduced fiber over $t=1$. 
We assume that $Y$ is 
a normal crossing divisor in $X$.
Let $\partial_\cM:H^2_\cM(X\setminus Y,\Q(2))\to H^3_{\cM,Y}(X,\Q(2))$ be the boundary map arising from the localization
sequence of motivic cohomology groups.
Let
$c_B:H^3_{\cM,Y}(X,\Q(2))\to H^3_Y(X,\Q(2))\cap H^{0,0}$.
be the Betti realization map. 
Let 
\[\partial:=c_B\circ\partial_\cM:H^2_\cM(X\setminus Y,\Q(2))\lra
H^3_Y(X,\Q(2))\cap H^{0,0}
\] 
be the composition, which we call the boundary map.
There is a natural injection
\[
T:=\Coker[T_1-1:R^1f_*\Q(2)\to R^1f_*\Q(1)]\hra
H^3_Y(X,\Q(2))
\]
where $T_1$ denotes the local monodromy at $t=1$.
One has
$T\cap H^{0,0}=H^3_Y(X,\Q(2))\cap H^{0,0}$.
The ring $R$ acts on $T$ and hence on $T\cap H^{0,0}$.
By the last condition of HG fibrations (see \S \ref{HG-defn}),
$T(e)\cong E$ as $E$-module and the Hodge type is $(0,0)$.
Therefore $T(e)\cap H^{0,0}=T(e)\cong E$.
We write by
\begin{equation}\label{boundary}
\partial(e):H^2_\cM(X\setminus Y,\Q(2))\lra (T\cap H^{0,0})(e)=T(e)\cap H^{0,0}\cong E
\end{equation}
the composition of $\partial$ with the projection onto the $e$-part.

\begin{thm}\label{boundary-thm1}
Let $f$ be a HG fibration of Gauss type 
\[
X_t=f^{-1}(t):y^N=x^a(1-x)^b(1-tx)^{N-b},\quad 1\leq a,b<N,\,\gcd(N,a,b)=1
\]
with multiplication by $(\Q[\mu_N],e)$ such that the projection $e:\Q[\mu_N]\to E$
satisfies $ad/N,bd/N\not\in\Z$, $d:=\sharp\ker(e:\mu_N\to E^\times)$.
Suppose
\begin{equation}\label{boundary-thm1-eq1}\tag{$\ast$}
\gcd(N,a)=\gcd(N,b)=1.
\end{equation}
Then the map $\partial(e)$ in \eqref{boundary} is surjective.
\end{thm}
I don't know whether one can remove the assumption \eqref{boundary-thm1-eq1}
in the above.

\medskip

Before the proof of Theorem \ref{boundary-thm1} we first show the following lemmas.
\begin{lem}\label{boundary-lem1}
$X$ is a rational surface (without assumption \eqref{boundary-thm1-eq1}).
\end{lem}
\begin{proof}
Put $z:=1-tx$. Then
\[
y^N=x^a(1-x)^b(1-tx)^{N-b}\quad \Longleftrightarrow \quad
(y/z)^N=x^a(1-x)^bz^{-b}.
\]
Let $y_1:=y/z$ and $z_1:=(1-x)/z$. Then $y_1^N=x^az_1^b$.
Let $z_2:=z_1/y$ then $y_1^{N-b}=x^az_2^b$.
If $N-b>b$, let $z_3:=z_2/y$ and then $y_1^{N-2b}=x^az_3^b$.
If $N-b<b$, let $y_2:=y/z_2$ and then $y_2^{N-b}=x^az_3^{2b-N}$.
Continuing this argument, we finally have
a surface $y_0^d=x^az_0^d$, $d:=\gcd(N,b)$ which is birational to $X$.
Note $\gcd(N,a,b)=\gcd(d,a)=1$. Apply the same argument for
the variables $y_0$ and $x$, we then have a surface
$y_0^\prime=x_0z_0^d$ which is birational to $X$. Therefore $X$ is a rational surface.
\end{proof}
\begin{lem}\label{boundary-lem2}
Let $\NS(X)$ be the Neron-Severi group.
The $e$-part $\NS(X)(e)$ is generated fibral divisors and a section
of $f$ (without assumption \eqref{boundary-thm1-eq1}).
Here we say a divisor $D$ is fibral if $f(D)$ is a point.
\end{lem}
\begin{proof}
Let $S:=\P^1\setminus \{0,1,\infty\}$ and $U:=f^{-1}(S)$.
Then 
\[H^2(X)/\langle\mbox{fibral divisors}\rangle\cong W_2H^2(U)\]
and there is an exact sequence
\[
0\lra H^1(S,R^2f_*\Q)\lra H^2(U,\Q)\lra H^2(X_t,\Q)\lra0.
\]
Since the last term is spanned by the image of the cycle class of a section, 
it is enough to show $W_2H^1(S,R^1f_*\Q)=0$. Let $j:S\hra \P^1$. 
Since there is an exact sequence
\[
\xymatrix{
0\ar[r]&H^1(\P^1,j_*R^1f_*\Q)\ar[r]&H^1(S,R^1f_*\Q)\ar[r]&
H^0(\P^1,R^1j_*R^1f_*\Q)\ar[r]&0\\
&W_2H^1(\P^1,R^1f_*\Q)\ar@{=}[u]
}
\]
it is enough to show that 
$H^1(S,R^1f_*\Q)(e)\to
H^0(\P^1,R^1j_*R^1f_*\Q)(e)$ is injective, or equivalently
\begin{equation}\label{boundary-lem2-eq1}
\dim H^1(S,R^1f_*\Q)(e)\leq \dim H^0(\P^1,R^1j_*R^1f_*\Q)(e).
\end{equation}
Since $H^0(S,R^1f_*\Q)(e)=0$, one has
\[
\dim H^1(S,R^1f_*\Q)(e)=-\chi(S,(R^1f_*\Q)(e))=-\chi(S,\Q)
\dim(H^1(X_t)(e))=2[E:\Q]
\]
by the second condition of HG fibration in \S \ref{HG-defn}.
On the other hand, letting $T_P$ denotes the local monodromy at $P$, 
\[
H^0(\P^1,R^1j_*R^1f_*\Q)(e)\cong 
\bigoplus_{P=0,1,\infty}\Coker[T_p-1:H^1(X_t)(e)\to H^1(X_t)(e)].
\]
By Lemma \ref{Gauss-lem3} the eigenvalues of each $T_P$ are known.
In particular both of $T_0$ and $T_1$ have eigenvalue $1$. This implies
$\dim H^0(\P^1,R^1j_*R^1f_*\Q)(e)\geq 2[E:\Q]$.
Thus \eqref{boundary-lem2-eq1} follows.
\end{proof}

We prove Theorem \ref{boundary-thm1}.
There are the localization sequences of the motivic cohomology groups
and the Deligne-Beilinson cohomology groups which sit in a commutative diagram
\[
\xymatrix{
H^2_\cM(X\setminus Y,\Q(2))\ar[r]^{\partial_\cM}\ar[d]_{\reg_{X\setminus Y}}&
H^3_{\cM,Y}(X,\Q(2))\ar[r]^i\ar[d]_{\reg_X^Y}& H^3_\cM(X,\Q(2))\ar[d]^{\reg_X}_\cong\\
H^2_\cD(X\setminus Y,\Q(2))\ar[r]&
H^3_{\cD,Y}(X,\Q(2))\ar[d]_{c_\cD}\ar[r]& H^3_\cD(X,\Q(2))\\
&H^3_Y(X,\Q(2))\cap H^{0,0}.
}
\]
where $c_\cD$ is the canonical surjective map, and
the bijectivity of $\reg_X$ follows from the fact that
$X$ is a smooth projective rational surface (Lemma \ref{boundary-lem1}).
Note, $c_B=c_\cD\circ\reg^Y_X$ is the Betti realization map, and hence
$\partial=c_\cD\circ\reg_X^Y\circ\partial_\cM$ is the boundary map as above.
Our goal is to show that there is a subspace $W\subset 
H^3_{\cM,Y}(X,\Q(2))$ such that $i(W)=0$ and
$W$ is onto $T(e)$ by $c_\cD\circ\reg_X^Y$. 
Let $Y=\sum Y_i$ be the irreducible decomposition, and $T\subset Y$ the singular locus.
Then there is the canonical isomorphism
\[
H^3_{\cM,Y}(X,\Q(2))\cong\ker\left[\bigoplus_i\C(Y_i)^\times\ot\Q\os{\mathrm{div}}{\lra}
\bigoplus_{P_i\in T}\Q\cdot P_i\right]
\]
where $\mathrm{div}$ is the map of divisor.
For $f\in \C(Y_i)^\times\ot\Q$, we denote by 
\[
[f,Y_i]\in\bigoplus_i\C(Y_i)^\times\ot\Q
\]
an element of placed in the component of $Y_i$.

To show Theorem \ref{boundary-thm1} we first describe $Y$ in detail. 
Let $\O=\C[[t-1]]$ and
\[\hat{g}:X^*:=\Spec \O[x,y]/(y^N-x^a(1-x)^b(1-tx)^{N-b})\lra\Spec \O.
\]
The surface $X^*$ has two isolated singularities $(x,y,t)=(0,0,1),(1,0,1)$.
Let $X_0\to X^*$ be the blow-up at $(x,y,t)=(1,0,1)$, and $U\subset X_0$
an affine open set such that
\[
U=\Spec\O[x,y_0,t_0]/(y^N_0-x^a(1+t_0x)^{N-b})\subset X_0
\]
where the morphism given by $y_0=y/(1-x)$ and $t_0=(1-t)/(1-x)$.
Let $D\subset X_0$ be the proper transform of the
central fiber of $\hat{f}_0$ over $t=1$, and 
$E$ the exceptional curve of the blow-up: 
\[
D\cap U=\{t_0=0,\,y^N_0=x^a\}\cong\Spec\O[x,y_1]/(y^N_0-x^a),
\]
\[
E\cap U=\{x=1,\,y^N_0=(1+t_0)^{N-b}\}\cong\Spec\O[y_0,t_0]/(y^N_0-(1+t_0)^{N-b}).
\]
By the assumption \eqref{boundary-thm1-eq1}, the curves $D$ and $E$ are irreducible.
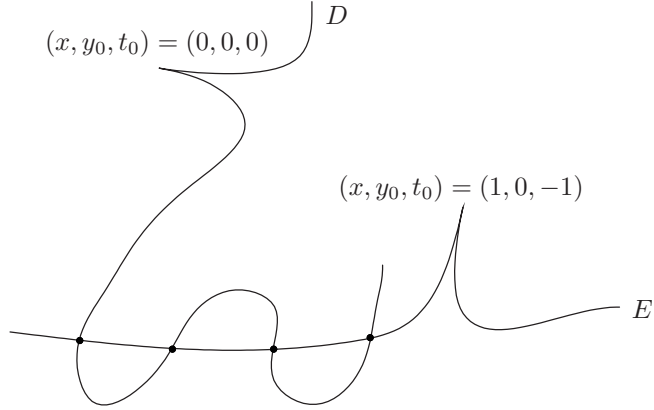
\begin{center}
%\input{K2.pic3.tex}
%WinTpicVersion4.30c
{\unitlength 0.1in%
\begin{picture}( 32.1300, 21.1600)(  2.0000,-22.2300)%
% SPLINE 2 0 3 0 Black White
% 9 200 1840 1990 1890 2220 1830 2510 1420 2570 1190 2550 1700 2570 1740 3370 1710 3390 1710
% 
\special{pn 8}%
\special{pa 200 1840}%
\special{pa 232 1844}%
\special{pa 263 1848}%
\special{pa 327 1856}%
\special{pa 358 1859}%
\special{pa 422 1867}%
\special{pa 453 1871}%
\special{pa 485 1874}%
\special{pa 516 1878}%
\special{pa 548 1882}%
\special{pa 580 1885}%
\special{pa 612 1889}%
\special{pa 643 1892}%
\special{pa 675 1895}%
\special{pa 707 1899}%
\special{pa 738 1902}%
\special{pa 802 1908}%
\special{pa 834 1910}%
\special{pa 865 1913}%
\special{pa 897 1916}%
\special{pa 929 1918}%
\special{pa 961 1921}%
\special{pa 1025 1925}%
\special{pa 1056 1927}%
\special{pa 1088 1928}%
\special{pa 1120 1930}%
\special{pa 1184 1932}%
\special{pa 1216 1934}%
\special{pa 1248 1934}%
\special{pa 1312 1936}%
\special{pa 1408 1936}%
\special{pa 1441 1935}%
\special{pa 1473 1935}%
\special{pa 1569 1932}%
\special{pa 1602 1930}%
\special{pa 1666 1926}%
\special{pa 1699 1924}%
\special{pa 1763 1918}%
\special{pa 1796 1915}%
\special{pa 1828 1912}%
\special{pa 1861 1908}%
\special{pa 1893 1904}%
\special{pa 1926 1899}%
\special{pa 1959 1895}%
\special{pa 1991 1890}%
\special{pa 2024 1884}%
\special{pa 2056 1879}%
\special{pa 2089 1872}%
\special{pa 2120 1865}%
\special{pa 2151 1856}%
\special{pa 2181 1846}%
\special{pa 2209 1835}%
\special{pa 2237 1822}%
\special{pa 2262 1807}%
\special{pa 2287 1791}%
\special{pa 2310 1773}%
\special{pa 2332 1753}%
\special{pa 2352 1731}%
\special{pa 2372 1708}%
\special{pa 2391 1683}%
\special{pa 2408 1657}%
\special{pa 2425 1628}%
\special{pa 2441 1598}%
\special{pa 2456 1567}%
\special{pa 2470 1533}%
\special{pa 2484 1498}%
\special{pa 2497 1462}%
\special{pa 2509 1423}%
\special{pa 2521 1383}%
\special{pa 2532 1343}%
\special{pa 2543 1305}%
\special{pa 2552 1269}%
\special{pa 2559 1238}%
\special{pa 2565 1213}%
\special{pa 2569 1195}%
\special{pa 2571 1187}%
\special{pa 2570 1189}%
\special{pa 2568 1199}%
\special{pa 2564 1216}%
\special{pa 2553 1271}%
\special{pa 2548 1306}%
\special{pa 2542 1345}%
\special{pa 2537 1387}%
\special{pa 2532 1431}%
\special{pa 2529 1476}%
\special{pa 2527 1521}%
\special{pa 2528 1566}%
\special{pa 2531 1609}%
\special{pa 2536 1649}%
\special{pa 2545 1685}%
\special{pa 2557 1717}%
\special{pa 2572 1744}%
\special{pa 2590 1766}%
\special{pa 2610 1785}%
\special{pa 2633 1800}%
\special{pa 2657 1811}%
\special{pa 2683 1820}%
\special{pa 2711 1825}%
\special{pa 2741 1828}%
\special{pa 2772 1828}%
\special{pa 2804 1826}%
\special{pa 2838 1822}%
\special{pa 2872 1816}%
\special{pa 2907 1809}%
\special{pa 2943 1801}%
\special{pa 2980 1792}%
\special{pa 3016 1782}%
\special{pa 3053 1772}%
\special{pa 3091 1762}%
\special{pa 3128 1752}%
\special{pa 3165 1743}%
\special{pa 3201 1734}%
\special{pa 3237 1726}%
\special{pa 3273 1720}%
\special{pa 3308 1715}%
\special{pa 3341 1711}%
\special{pa 3374 1710}%
\special{pa 3390 1710}%
\special{fp}%
% SPLINE 2 0 3 0 Black White
% 20 1780 110 1700 410 990 460 1180 510 1430 750 1150 1070 850 1380 700 1620 550 1970 640 2210 970 2040 1180 1720 1380 1620 1580 1700 1580 2090 1790 2220 2020 2060 2120 1690 2150 1510 2150 1490
% 
\special{pn 8}%
\special{pa 1780 110}%
\special{pa 1780 182}%
\special{pa 1779 217}%
\special{pa 1776 251}%
\special{pa 1770 284}%
\special{pa 1762 315}%
\special{pa 1751 344}%
\special{pa 1736 371}%
\special{pa 1717 395}%
\special{pa 1692 416}%
\special{pa 1663 434}%
\special{pa 1629 449}%
\special{pa 1591 461}%
\special{pa 1550 470}%
\special{pa 1507 477}%
\special{pa 1461 482}%
\special{pa 1415 486}%
\special{pa 1367 487}%
\special{pa 1320 487}%
\special{pa 1273 486}%
\special{pa 1227 484}%
\special{pa 1184 482}%
\special{pa 1142 478}%
\special{pa 1104 474}%
\special{pa 1070 471}%
\special{pa 1041 467}%
\special{pa 1016 464}%
\special{pa 997 461}%
\special{pa 984 459}%
\special{pa 979 458}%
\special{pa 981 458}%
\special{pa 991 460}%
\special{pa 1009 464}%
\special{pa 1035 469}%
\special{pa 1067 476}%
\special{pa 1102 485}%
\special{pa 1139 496}%
\special{pa 1177 509}%
\special{pa 1214 524}%
\special{pa 1249 541}%
\special{pa 1282 561}%
\special{pa 1313 581}%
\special{pa 1342 604}%
\special{pa 1367 627}%
\special{pa 1388 651}%
\special{pa 1406 676}%
\special{pa 1419 701}%
\special{pa 1427 727}%
\special{pa 1430 752}%
\special{pa 1428 778}%
\special{pa 1421 803}%
\special{pa 1409 828}%
\special{pa 1393 852}%
\special{pa 1374 877}%
\special{pa 1352 901}%
\special{pa 1327 926}%
\special{pa 1301 949}%
\special{pa 1273 973}%
\special{pa 1244 997}%
\special{pa 1214 1020}%
\special{pa 1185 1043}%
\special{pa 1156 1065}%
\special{pa 1128 1088}%
\special{pa 1101 1110}%
\special{pa 1075 1132}%
\special{pa 1050 1155}%
\special{pa 1026 1177}%
\special{pa 1003 1199}%
\special{pa 981 1221}%
\special{pa 959 1244}%
\special{pa 919 1290}%
\special{pa 899 1314}%
\special{pa 881 1338}%
\special{pa 862 1363}%
\special{pa 844 1388}%
\special{pa 827 1414}%
\special{pa 776 1495}%
\special{pa 760 1523}%
\special{pa 743 1550}%
\special{pa 709 1606}%
\special{pa 692 1633}%
\special{pa 674 1660}%
\special{pa 657 1687}%
\special{pa 639 1714}%
\special{pa 607 1770}%
\special{pa 593 1798}%
\special{pa 580 1828}%
\special{pa 569 1858}%
\special{pa 561 1889}%
\special{pa 554 1921}%
\special{pa 551 1955}%
\special{pa 550 1990}%
\special{pa 552 2027}%
\special{pa 558 2063}%
\special{pa 567 2098}%
\special{pa 578 2130}%
\special{pa 593 2160}%
\special{pa 610 2184}%
\special{pa 630 2204}%
\special{pa 653 2216}%
\special{pa 679 2222}%
\special{pa 706 2222}%
\special{pa 736 2216}%
\special{pa 766 2205}%
\special{pa 797 2190}%
\special{pa 828 2171}%
\special{pa 858 2149}%
\special{pa 888 2124}%
\special{pa 917 2098}%
\special{pa 944 2070}%
\special{pa 969 2041}%
\special{pa 991 2012}%
\special{pa 1012 1984}%
\special{pa 1030 1955}%
\special{pa 1047 1927}%
\special{pa 1062 1899}%
\special{pa 1078 1871}%
\special{pa 1093 1844}%
\special{pa 1108 1818}%
\special{pa 1124 1792}%
\special{pa 1141 1767}%
\special{pa 1160 1743}%
\special{pa 1180 1720}%
\special{pa 1203 1698}%
\special{pa 1228 1678}%
\special{pa 1255 1660}%
\special{pa 1284 1645}%
\special{pa 1314 1633}%
\special{pa 1347 1624}%
\special{pa 1380 1620}%
\special{pa 1415 1620}%
\special{pa 1451 1625}%
\special{pa 1485 1634}%
\special{pa 1516 1647}%
\special{pa 1544 1664}%
\special{pa 1567 1684}%
\special{pa 1584 1707}%
\special{pa 1594 1733}%
\special{pa 1599 1762}%
\special{pa 1599 1792}%
\special{pa 1596 1825}%
\special{pa 1591 1858}%
\special{pa 1577 1928}%
\special{pa 1572 1962}%
\special{pa 1568 1997}%
\special{pa 1568 2030}%
\special{pa 1572 2062}%
\special{pa 1581 2093}%
\special{pa 1597 2121}%
\special{pa 1618 2147}%
\special{pa 1643 2169}%
\special{pa 1672 2188}%
\special{pa 1703 2203}%
\special{pa 1736 2214}%
\special{pa 1769 2219}%
\special{pa 1802 2219}%
\special{pa 1834 2214}%
\special{pa 1865 2204}%
\special{pa 1895 2189}%
\special{pa 1922 2171}%
\special{pa 1948 2149}%
\special{pa 1972 2125}%
\special{pa 1994 2099}%
\special{pa 2013 2071}%
\special{pa 2030 2042}%
\special{pa 2045 2013}%
\special{pa 2057 1983}%
\special{pa 2067 1953}%
\special{pa 2076 1922}%
\special{pa 2084 1891}%
\special{pa 2090 1860}%
\special{pa 2096 1828}%
\special{pa 2101 1796}%
\special{pa 2107 1765}%
\special{pa 2112 1733}%
\special{pa 2118 1701}%
\special{pa 2124 1670}%
\special{pa 2131 1638}%
\special{pa 2138 1607}%
\special{pa 2144 1575}%
\special{pa 2148 1544}%
\special{pa 2150 1512}%
\special{pa 2150 1490}%
\special{fp}%
% STR 2 0 3 0 Black White
% 4 1910 80 1910 180 5 0 0 0
% $D$
\put(19.1000,-1.8000){\makebox(0,0){$D$}}%
% STR 2 0 3 0 Black White
% 4 3510 1620 3510 1720 5 0 0 0
% $E$
\put(35.1000,-17.2000){\makebox(0,0){$E$}}%
% STR 2 0 3 0 Black White
% 4 2560 990 2560 1090 5 0 0 0
% $(x,y_0,t_0)=(1,0,-1)$
\put(25.6000,-10.9000){\makebox(0,0){$(x,y_0,t_0)=(1,0,-1)$}}%
% STR 2 0 3 0 Black White
% 4 970 230 970 330 5 0 0 0
% $(x,y_0,t_0)=(0,0,0)$
\put(9.7000,-3.3000){\makebox(0,0){$(x,y_0,t_0)=(0,0,0)$}}%
% DOT 0 0 3 0 Black White
% 2 1580 1930 1580 1930
% 
\special{pn 4}%
\special{sh 1}%
\special{ar 1580 1930 16 16 0  6.28318530717959E+0000}%
\special{sh 1}%
\special{ar 1580 1930 16 16 0  6.28318530717959E+0000}%
% DOT 0 0 3 0 Black White
% 2 1050 1930 1050 1930
% 
\special{pn 4}%
\special{sh 1}%
\special{ar 1050 1930 16 16 0  6.28318530717959E+0000}%
\special{sh 1}%
\special{ar 1050 1930 16 16 0  6.28318530717959E+0000}%
% DOT 0 0 3 0 Black White
% 2 565 1885 565 1885
% 
\special{pn 4}%
\special{sh 1}%
\special{ar 565 1885 16 16 0  6.28318530717959E+0000}%
\special{sh 1}%
\special{ar 565 1885 16 16 0  6.28318530717959E+0000}%
% DOT 0 0 3 0 Black White
% 2 2085 1870 2085 1870
% 
\special{pn 4}%
\special{sh 1}%
\special{ar 2085 1870 16 16 0  6.28318530717959E+0000}%
\special{sh 1}%
\special{ar 2085 1870 16 16 0  6.28318530717959E+0000}%
\end{picture}}%
\end{center}
Let $X\to X_0$ be a desingularization, then 
the fiber over $t=1$ is $Y$.
Hence there is an embedding of the nomalization of $D\cup E$ into $Y$.
\begin{itemize}
\item $\sigma D=D$, 
$\sigma E=E$ for any automorphism $\sigma\in \mu_N$.
\item $D\cap U$ has a singular point $(x,y_0,t_0)=(0,0,0)$ unless $a=1$.
Let $\iota:D'\to D$ be the normalization, then $D'\cong \P^1$ and 
$\iota^{-1}(D\cap U)\cong\bA^1$.
\item $E\cap U$ has a singular point $(x,y_0,t_0)=(1,0,-1)$ unless $b=N-1$.
Let $\iota:E'\to E$ be the normalization, then 
$E'\cong\P^1$ and $\iota^{-1}(E\cap U)\cong \bA^1$.
\item
$D$ and $E$ intersect transversally and $D\cap E\subset U$.
Moreover $U$ is regular at each point of $D\cap E$.
\end{itemize}
We denote by $u$ and $v$ the affine coordinates of $\iota^{-1}(D\cap U)$ 
and $\iota^{-1}(E\cap U)$ 
respectively such
that
\[
(y_0,x)|_D=(u^a,u^N),\quad
(y_0,1+t_0)|_E=(v^{N-b},v^N)
\]
The intersection points of $D\cap E$ consist of $\{u=\zeta\mid \zeta\in\mu_N\}$
or $\{v=\zeta\mid \zeta\in\mu_N\}$.
A point $u=\zeta$ corresponds to $v=\zeta'$ if $\zeta^a=(\zeta')^{N-b}=(\zeta')^{-b}$.
Thinking of $D'$ and $E'$ being components of $Y=f^{-1}(1)$,
we consider elements
\[
\Xi(\zeta_1,\zeta_2):=\left[\frac{u-\zeta_1}{u-\zeta_2},D'\right]-
\left[\frac{v-\zeta_1^{-a/b}}{v-\zeta_2^{-a/b}},E'\right]\in H^3_{\cM,Y}(X,\Q(2))
\]
for $\zeta_1,\zeta_2\in\mu_N$
in the motivic cohomology supported on $Y$.
Define $W\subset H^3_{\cM,Y}(X,\Q(2))$ to be the subspace generated by 
$\Xi(\zeta_1,\zeta_2)$'s.

We first show $i(W)=0$.
Note $H^3_\cM(X,\Q(2))\cong (\C^\times\ot \NS(X))\ot\Q$ since $X$ is a 
rational surface (Lemma \ref{boundary-lem1}). Giving generators $F_n$'s of
$\NS(X)\ot\Q$ which intersect with $D'\cup E'$ properly outside the points
$u=\zeta_i$ or $v=\zeta_i^{-a/b}$,
one has
\[
i(\Xi(\zeta_1,\zeta_2))=\sum_n c_n\ot F_n,
\]
\[
c_n:=
\prod_{P\in F_n\cap D'} \left(\frac{u-\zeta_1}{u-\zeta_2}\bigg|_P\right)^{m_P}
\times
\prod_{Q\in F_n\cap E'} \left(\frac{v-\zeta_1^{-a/b}}{v-\zeta_2^{-a/b}}\bigg|_Q\right)^{-m_Q}
\in\C^\times
\]
where $m_P,m_Q$ denote the intersection numbers.
By Lemma \ref{boundary-lem2}, the $e$-part 
$\NS(X)(e)$ is generated by fibral divisors and a section. 
If $F_n$ is a section of $x=\infty$, then $P$ and $Q$ are the points defined by
$u=\infty$ and $v=\infty$ respectively. Therefore $c_n$ is torsion.
Suppose that $F_n$ is an irreducible fibral divisor which is not $D'$ or $E'$.
Then the intersection points of $D'$ and $F_n$ are at most $u=0$ or $u=\infty$.
Therefore the first term of $c_n$ is torsion. 
In the same way, the second term is also torsion, and hence so is $c_n$.
Let $F_n=E'$. Replace $E'$ with $E^{\prime\prime}=E'-\mathrm{div}(x-1)$.
Let $E^{\prime\prime}_0$ be the image of $E^{\prime\prime}$ in $X_0$.
Then any component of $E^{\prime\prime}_0$ is neither $D$ or $E$.
Moreover it intersects with $D\cap U$ (resp. $E\cap U$) 
at most at the singular point $(y_0,x)=(0,0)$ (resp. $(t_0,y_0)=(-1,0)$).
Therefore the intersection points of $E^{\prime\prime}\cap D'$ or 
$E^{\prime\prime}\cap E'$ are
at most $u=0,\infty$ or $v=0,\infty$. Hence $c_n$ is torsion.
Finally let $F_n=D'_k$. Then replace $D'_k$ with $D'_k-\mathrm{div}(t-1)$, a fibral divisor
without component $D'_k$. Hence this is reduced to the above.
This completes the proof of $i(\Xi(\zeta_1,\zeta_2))=0$, and hence $i(W)=0$.

There remains to show that $W$ is onto $T(e)$.
Let $D^*:=D'\setminus\{u^N=1\}$, and let
\[
T\cong H_{D'+E'}^3(X,\Q(2))\cap H^{0,0}\os{\subset}{\lra} H^1(D^*,\Q(1))
\os{\subset}{\lra}\bigoplus_{\zeta\in\mu_N}\Q\cdot(u=\zeta)
\]
be the composition of the Poincare residue maps.
An automorphism
$\sigma\in\mu_N$ such that $\sigma(y)=\zeta_N y$ acts on the last term
by $\sigma(u)=\zeta_N^{1/a} u$.
The above map induces an isomorphism $T(e)\cong H^1(D^*,\Q(1))(e)$ on the $e$-part.
Under this identification, one directly has
\[
\partial(\Xi(\zeta_1,\zeta_2))=(u=\zeta_1)-(u=\zeta_2).
\]
This means that $W$ is onto $H^1(D^*,\Q(1))$ 
and hence onto $T(e)\cong H^1(D^*,\Q(1))(e)$. This completes the proof. \qed

\begin{prob}\label{boundary-prob}
Find explicit descriptions of the $K_2$-symbols in $K_2(X\setminus Y)$ 
constructed in Theorem \ref{boundary-thm1}.
\end{prob}
If $f$ is a HG fibration defined by
\[
y^N=x(1-x)^{N-1}(1-tx),
\]
then one finds $K_2$-symbols
\begin{equation}\label{boundary-eq1}
\left\{\frac{y-\zeta_1(1-x)}{y-\zeta_2(1-x)},\frac{(1-x)^2}{x^2(1-t)}\right\}
\in K_2(X\setminus Y),\quad
\zeta_i\in\mu_N,
\end{equation}
and shows that their boundary span $T(e)$ (hence we don't need 
Theorem \ref{boundary-thm1} in this case).  
I don't know how to find such symbols for general $y^N=x^a(1-x)^b(1-tx)^{N-b}$.
\begin{cor}\label{boundary-cor1}
Let $f$ be a HG fibration of Gauss type as in Theorem \ref{boundary-thm1}.
Let
$\Res:\vg(X,\Omega^2_X(\log Y))\to H^\dR_1(Y)\cong H^B_1(Y,\C)$ 
be the Poincare residue map at $t=1$.
Let
\[
\dlog:H^2_\cM(X\setminus Y,\Q(2))\lra \vg(X,\Omega^2_X(\log Y))(e)
=\bigoplus_{n\in I_e}\C\cdot
\frac{dt}{t-1}\omega_n
\]
be the dlog map (see Lemma \ref{Gauss-lem2} for the right hand side).
Then the dlog map is onto a set of $2$-forms
\[
V:=\left\{\sum_{n\in I_e} \l_n\left(\frac{dt}{t-1}\omega_n\right)\mbox{ s.t. }
\sum_{n\in I_e} \l_n\Res\left(\frac{dt}{t-1}\omega_n\right)\in H^B_1(Y,\Q)(e)\right\}
\]
where $\omega_n$ and $I_e$ are as in Lemma \ref{Gauss-lem1}.
\end{cor}
We note 
\begin{equation}\label{boundary-cor1-eq1}
\sum_{n\in I_e} \l_n\Res\left(\frac{dt}{t-1}\omega_n\right)\in H^B_1(Y,\Q)(e)
\quad\Longleftrightarrow\quad \sum_{n\in I_e}\l_n\zeta^n\in\Q,\,\forall\zeta\in\mu_N.
\end{equation}
\begin{proof}
Obviously  $\Image(\dlog)\subset V$. Since $\dim_\Q\Image(\dlog)=[E:\Q]$ 
by Theorem \ref{boundary-thm1}, it is enough to show $\dim_\Q V\leq [E:\Q]$.
However this is immediate from \eqref{boundary-cor1-eq1}.
\end{proof}
\subsection{Main Theorem}\label{m-gauss-sect2}
Let $f$ be a HG fibration of Gauss type 
\[
X_t=f^{-1}(t):y^N=x^a(1-x)^b(1-tx)^{N-b},\quad 1\leq a,b<N,\,\gcd(N,a,b)=1
\]
with multiplication by $(\Q[\mu_N],e)$ such that the projection $e:\Q[\mu_N]\to E$
satisfies $ad/N,bd/N\not\in\Z$, $d:=\sharp\ker(e:\mu_N\to E^\times)$.
Let $\xi\in H^2_\cM(X\setminus Y,\Q(2))(e)$ be an element of the $e$-part, and
let 
\begin{equation}\label{m-gauss-sect1-eq1}
\dlog(\xi)=\sum_{n\in I_e} \l_n\left(\frac{dt}{t-1}\omega_n\right).
\end{equation}
Note $\l_n$'s satisfy the condition \eqref{boundary-cor1-eq1}.
Conversely
if $\gcd(N,a)=\gcd(N,b)=1$, then it follows from Theorem \ref{boundary-thm1} that,
for any $\l_n$'s satisfying \eqref{boundary-cor1-eq1} there exists $\xi$ such that
\eqref{m-gauss-sect1-eq1} holds. 
\begin{thm}\label{m-gauss-thm1}
Suppose $a\ne b$. 
Let $\gamma_0=(1-\sigma)u_0$
and $\gamma_1=(1-\sigma)u_1$ be the homology cycles as in
Lemma \ref{Gauss-lem3}. Write $a_n:=\{an/N\}$, $b_n:=\{bn/N\}$, $z:=(1-t)^{-1}$ and 
\[
{}_4F^{(n)}_3(t):={}_4F_3\left({a_n+1,b_n+1,1,1\atop 2,2,2};t\right).
\]
Then
\begin{align*}
&\frac{1}{2\pi\sqrt{-1}}\langle\reg(\xi)\mid \gamma_1\rangle\\
&=\sum_{n\in I_e} (1-\zeta_N^n) \l_n
[2\psi(1)-\psi(a_n)-\psi(b_n)-\log(1-t)
-a_nb_n(1-t){}_4F_3^{(n)}(1-t)]\\
&=\sum_{n\in I_e} (1-\zeta_N^n) \l_n[a_n^{-1}C_{a_n,b_n}(-z)^{a_n}\,F_{a_n,b_n}(z)
+b_n^{-1}C_{b_n,a_n}(-z)^{b_n}\,F_{b_n,a_n}(z)]
\end{align*}
and
\[
\langle\reg(\xi)\mid \gamma_0\rangle
=\sum_{n\in I_e}(1-\zeta_N^n) \l_n[a_n^{-1}B_{a_n,b_n}z^{a_n}\,F_{a_n,b_n}(z)
+b_n^{-1}B_{b_n,a_n}z^{b_n}\,F_{b_n,a_n}(z)],
\]
where $B_{a,b}$, $C_{a,b}$ and $F_{a,b}(z)$ are 
as in Theorems \ref{m-fermat-thm2} and \ref{m-fermat-thm3}.
\end{thm}
\begin{proof}
The same proof as Theorems \ref{m-fermat-thm1},
\ref{m-fermat-thm3} and \ref{m-fermat-thm2}.
\end{proof}
It is worth noting that Theorem \ref{m-gauss-thm1} is proven without knowledge
of explicit description of the $K_2$-symbol $\xi$ (cf. Problem \ref{boundary-prob}).

\begin{conj}\label{m-gauss-thm2}
The first equality in Theorem \ref{m-gauss-thm1} is valid even when $a=b$.
\end{conj}

\section{Real regulators of $K_2$ of Elliptic fibrations and the Beilinson conjecture}
\label{bei-sect}
For an elliptic curve $E$ over $\R$ we discuss the {\it real regulator map}
\[
\reg_\R:H^2_\cM(E,\Z(2))\lra \R
\]
which is defined in the following way.
Let $F_\infty:E(\C)\to E(\C)$ be the infinite Frobenius map of real manifold.
We denote by ``$F_\infty=1$'' (resp. ``$F_\infty=-1$'') the fixed part
(resp. anti-fixed part). Then $\reg_\R$ is defined as the composition 
\begin{align*}
H^2_\cM(E,\Z(2))&\os{\reg}{\lra} 
\Hom(H_1^B(E(\C),\Z)^{F_\infty=-1},\C/\Z(2))\\
&\lra\Hom(H_1^B(E(\C),\Q)^{F_\infty=-1},\R(1))\\
&\os{\cong}{\lra}\R(1)\\
&\os{\cong}{\lra}\R
\end{align*}
where the 2nd arrow is given by 
the projection $\C/\Z(2)\to\R(1)=\sqrt{-1}\R$, the 3rd given by a
fixed base of $H_1^B(E(\C),\Q)^{F_\infty=-1}
\cong \Q$, and the 4th arrow given by multiplication by $(2\pi\sqrt{-1})^{-1}$.
\begin{conj}[Beilinson conjecture for an elliptic curve over $\Q$, 
cf.\,\cite{schneider}, \cite{DW}]
\label{bei-conj}
Let $E$ be an elliptic curve over $\Q$, and
$L(E,s)$ the motivic $L$-function of $E$.
Then there is an integral element $\xi\in H^2_\cM(E,\Z(2))$ in the sense of Scholl \cite{Scholl}
such that
\[
\reg_\R(\xi)\sim_{\Q^\times}\pi^{-2}L(E,2)
\]
where $x\sim_{\Q^\times}y$ means $xy\ne0$ and $xy^{-1}\in \Q^\times$.
\end{conj}
Beilinson further conjectures that the space $H^2_\cM(E,\Q(2))_\Z$ of
integral elements is 1-dimensional, spanned by $\xi$ in the above,
though we will not discuss this issue.

\subsection{}\label{bei-sect1}
Let $f:X\to \P^1$ be the Legendre family of elliptic curves given by an affine equation
\[
X_t=f^{-1}(t):y^2=x(1-x)(1-tx).
\]
Consider a $K_2$-symbol
\begin{equation}\label{bei-thm1-eq0}
\xi:=\left\{\frac{y-x+1}{y+x-1},\frac{(x-1)^2}{x^2(t-1)}
\right\}\in K_2(X\setminus Y),\quad Y:=f^{-1}(1).
\end{equation}
One immediately has
\[
\dlog\xi=\frac{dx}{y}\frac{dt}{t-1}.
\]
\begin{thm}\label{bei-thm1}
Write $\xi_t:=\xi|_{X_t}\in K_2(X_t)$ for $t\in\R\setminus\{0,1\}$.
\begin{enumerate}
\item[$(1)$]
If $t>0$, then
\[
\reg_\R(\xi_t)=\mathrm{Re}\left[-\log16+\log(1-t)+ 
\frac{1-t}{4}{}_4F_3\left({\frac{3}{2},\frac{3}{2},1,1\atop
2,2,2};1-t\right)\right].
\]
\item[$(2)$]
If $t<0$, then
\[
\reg_\R(\xi_t)=z^{\frac{1}{2}}{}_3F_2\left({\frac{1}{2},\frac{1}{2},\frac{1}{2}\atop
1,\frac{3}{2}};z\right),\quad z:=(1-t)^{-1}.
\]
\end{enumerate}
\end{thm}
We can prove Theorem \ref{bei-thm1} in a similar way to \S \ref{m-fermat-sect},
on noting the following.
\begin{description}
\item[Case $t>0$]  $H_1^B(X_t(\C),\Q)^{F_\infty=-1}$ is spanned by a homology cycle
$\delta_t$ going around the interval from $x=1$ to $x=t^{-1}$, and
\[
\int_{\delta_t}\frac{dx}{y}=
2\pi \sqrt{-1}F\left(\frac{1}{2},\frac{1}{2},1;1-t\right),
\]
\item[Case $t<0$] $\bysame$ $\gamma_t$ $\bysame$ from $x=0$ to $x=t^{-1}$, and
\[
\int_{\gamma_t}\frac{dx}{y}=
2\pi z^{\frac{1}{2}}F\left(\frac{1}{2},\frac{1}{2},1;z\right),\quad z=(1-t)^{-1}.
\]
\end{description}

\begin{cor}
Let $t\in\Q\setminus\{0,1\}$ such that $\xi_t$ is integral. Then we have an equivalence
\begin{align}
&\mbox{Beilinson's Conjecture \ref{bei-conj} for $X_t$}\quad\Longleftrightarrow\quad\\
&\pi^{-2}L(X_t,2)\sim_{\Q^\times}
\begin{cases}
\mathrm{Re}\left[
-\log16+\log(1-t)+ 
\frac{1-t}{4}{}_4F_3\left({\frac{3}{2},\frac{3}{2},1,1\atop
2,2,2};1-t\right)\right]&t>0\\
z^{\frac{1}{2}}{}_3F_2\left({\frac{1}{2},\frac{1}{2},\frac{1}{2}\atop
1,\frac{3}{2}};z\right)&t<0.\label{bei-thm1-eq1}
\end{cases}
\end{align}
\end{cor}

\begin{cor}
Beilinson's Conjecture \ref{bei-conj} is true for $X_{-3}$.
\end{cor}
\begin{proof}
Indeed Rogers and Zudilin shows  
\[
\frac{\pi^2}{12}\,{}_3F_2\left({\frac{1}{2},\frac{1}{2},\frac{1}{2}\atop \frac{3}{2},1};
\frac{1}{4}\right)=L(E_{24},2)
\]
where $E_{24}$ is an elliptic curve over $\Q$ of conductor $24$
(\cite{RZ} Theorem 2, p.399 and (6), p.386).
There is only one elliptic curve of conductor $24$ up to isogeny, and
$X_{-3}$ is the one. Hence \eqref{bei-thm1-eq1} holds.
\end{proof}
\begin{center}
\bf Numerical Examples
\end{center}
By definition of the symbol $\xi$ in \eqref{bei-thm1-eq0}, 
$\xi_t$ is integral
if $X_t$ does not have a multiplicative reduction at any
prime $p$ such that $\ord_p(1-t)\ne0$.
In more practical way, $\xi_t$ is integral if
\[
\ord_p(j(X_t))=\ord_p\left(\frac{256(t^2-t+1)^3}{t^2(1-t)^2}\right)\geq 0
\mbox{ for any $p$ s.t. }\ord_p(1-t)\ne0
\]
\[
\Longleftrightarrow\quad t=-1,-3,-7,-15, 2, 3, 5, 9, 17,
\frac{1}{2}, \frac{3}{2}, 
\frac{7}{8}, \frac{9}{8}, \frac{3}{4}, \frac{5}{4}, \frac{15}{16}, \frac{17}{16}.
\]
Put
\[
R_t:=\reg_\R(\xi_t)/(\pi^{-2}L(X_t,2)).
\]
Here is the list of numerical verification of Beilinson's Conjecture \ref{bei-conj}
for above $t$'s with the aid of MAGMA (digits of precision is at least  20).

\begin{center}
 \begin{tabular}{c|c|c|c|c|c|c|c|c|c|c|c|c|c|c|c|c|c|c}
 $t$&$-1$&$-3$&$-7$&$-15$&$2$&$3$&$5$&$9$&$17$\\
 \hline
$R_t$&
$8$&$6$&$\frac{7}{2}$&$\frac{15}{4}$&$-32$&$-24$&$-20$&$-18$&$-17$
\end{tabular}
 \end{center}

\begin{center}
 \begin{tabular}{c|c|c|c|c|c|c|c|c|c|c|c|c|c|c|c|c|c|}
 $t$&$\frac{1}{2}$&$\frac{3}{2}$&$\frac{7}{8}$&$\frac{9}{8}$
 &$\frac{3}{4}$&$\frac{5}{4}$&$\frac{15}{16}$&$\frac{17}{16}$\\
 \hline
$R_t$&
$-32$&$-48$&$-56$&$-48$&$-48$&$-40$&$-\frac{165}{2}$&$-68$
\end{tabular}
 \end{center}
\subsection{}
Let $f:X\to \P^1$ be en elliptic fibration defined by $3y^2=2x^3-3x^2+t$.
Put
\begin{equation}\label{bei-thm1-eq2}
\xi:=\left\{
\frac{y-x+1}{y+x-1},
\frac{1-t}{2(x-1)^3}
\right\}\in K_2(X\setminus Y),\quad Y:=f^{-1}(1).
\end{equation}
In a similar way to Theorem \ref{bei-thm1} we have the following theorem.
\begin{thm}
Let $t\in\R\setminus\{0,1\}$.
If $|t-1|<1$, then
\[
\reg_\R(\xi_t)=\log432-\log (1-t)
-\frac{5}{36}(1-t)\,{}_4F_3\left(
{\frac{7}{6},\frac{11}{6},1,1\atop 2,2,2};1-t
\right).
\]
If $|t-1|>1$, then 
\[
\reg_\R(\xi_t)=\pi^{-1}\left[
\frac{3}{2}B\left(\frac{1}{6},\frac{1}{6}\right)z^{\frac{1}{6}}\,{}_3F_2\left(
{\frac{1}{6},\frac{1}{6},\frac{1}{6}\atop \frac{1}{3},\frac{7}{6}};z
\right)
+\frac{3}{10}B\left(\frac{5}{6},\frac{5}{6}\right)z^{\frac{5}{6}}\,{}_3F_2\left(
{\frac{5}{6},\frac{5}{6},\frac{5}{6}\atop \frac{5}{3},\frac{11}{6}};z
\right)\right]
\]
where $z:=(1-t)^{-1}$.
\end{thm}
\begin{center}
\bf Numerical Examples
\end{center}
Let $n\geq 2$ be an integer and let $t=1-1/n$.
Then $j(X_t)=432n^2/(n-1)$ and
\[
\xi_t=\left\{
\frac{y-x+1}{y+x-1},
\frac{1}{2n(x-1)^3}
\right\}.
\]
Therefore if the denominator of $432n^2/(n-1)$ is prime to $6n$,
then $\xi_t$ is integral.
There exist infinitely many such $n$'s (e.g. $n\geq 2$ such that $n\equiv 0,2$ mod $6$).
\begin{center}
 \begin{tabular}{c|c|c|c|c|c|c|c|c|c|c|c|c|c|c|c|c|c|c|c}
 $n$&$2$&$3$&$4$&$5$&$6$&$7$&$8$&$9$&$10$&$11$\\
 \hline
$R_t$&
$72$&$\frac{486}{7}$&$81$&$\frac{135}{2}$&$\frac{4860}{67}$&$\frac{189}{2}$
&$\frac{1512}{19}$&$81$&$90$&$\frac{165}{2}$
\end{tabular}
 \end{center}
\begin{center}
 \begin{tabular}{c|c|c|c|c|c|c|c|c|c|c|c|c|c|c|c|c|c|c|c}
 $n$&$12$&$13$&$14$&$15$&$16$&$17$&$18$&$19$&$20$&$21$\\
 \hline
$R_t$&
$\frac{2673}{28}$&$\frac{13689}{176}$&$\frac{34398}{443}$
&$\frac{1701}{19}$&$\frac{405}{8}$
&$\frac{2601}{23}$&$\frac{2754}{29}$&$\frac{3249}{40}$&$\frac{171}{2}$
&$\frac{8505}{104}$
\end{tabular}
 \end{center}

\subsection{}
Let $f:X\to \P^1$ be en elliptic fibration defined by $y^2=x^3+(3x+4t)^2$.
Put
\begin{equation}\label{bei-thm1-eq3}
\xi:=\left\{
\frac{y-3x-4t}{-8t},
\frac{y+3x+4t}{8t}
\right\}\in K_2(X\setminus Y), \quad Y:=f^{-1}(1).
\end{equation}
\begin{thm}
Let $t\in\R\setminus\{0,1\}$.
If $0<|t|<1$, then
\[\reg_\R(\xi_t)=
\log27-\log t
-\frac{2t}{9}\,{}_4F_3\left(
{\frac{4}{3},\frac{5}{3},1,1\atop 2,2,2};t
\right).\]
If $|t|>1$, then
\[\reg_\R(\xi_t)=
\sqrt{3}\pi^{-1}
\left[
B\left(\frac{1}{3},\frac{1}{3}\right)
t^{-\frac{1}{3}}{}_3F_2\left({\frac{1}{3},\frac{1}{3},\frac{1}{3}
\atop \frac{2}{3},\frac{4}{3}};t^{-1}\right)
+\frac{1}{2}B\left(\frac{2}{3},\frac{2}{3}\right)
t^{-\frac{2}{3}}{}_3F_2\left({\frac{2}{3},\frac{2}{3},\frac{2}{3}
\atop \frac{4}{3},\frac{5}{3}};t^{-1}\right)
\right].
\]
\end{thm}
\begin{center}
\bf Numerical Examples
\end{center}
Let $n\geq 1$ be an integer, and let $t=\frac{1}{6n}$.
Then $j(X_t)=1296(4-27n)^3n/(6n-1)$ and
\[
\xi_t=\left\{
-\frac{3n}{4}\left(y-3x-\frac{2}{3n}\right),
\frac{3n}{4}\left(y+3x+\frac{2}{3n}\right)\right\}.
\]
Since the denominator of $j(X_t)$ is prime to $6n$, $\xi_t$ is integral for all $n\geq 1$.

\begin{center}
 \begin{tabular}{c|c|c|c|c|c|c|c|c|c|c|c|c|c|c|c|c|c|c|c}
 $n$&$1$&$2$&$3$&$4$&$5$&$6$&$7$&$8$&$9$&$10$\\
 \hline
$R_t$&
$\frac{405}{8}$&$\frac{891}{16}$&$\frac{1377}{20}$&$\frac{5589}{88}$
&$\frac{19575}{256}$&$\frac{135}{2}$&$\frac{54243}{776}$&$\frac{1269}{16}$
&$\frac{477}{8}$&$\frac{13275}{166}$
\end{tabular}
 \end{center}
\begin{center}
 \begin{tabular}{c|c|c|c|c|c|c|c|c|c|c|c|c|c|c|c|c|c|c|c}
 $n$&$11$&$12$&$13$&$14$&$15$&$16$&$17$&$18$&$19$&$20$\\
 \hline
$R_t$&
$\frac{70785}{1016}$&$\frac{5751}{64}$&$\frac{10647}{128}$
&$\frac{15687}{230}$&$\frac{20025}{248}$&$\frac{2565}{32}$
&$\frac{788103}{10172}$&$\frac{321}{4}$&$\frac{1101411}{14216}$&$\frac{80325}{872}$
\end{tabular}
 \end{center}
 
\subsection{Remark on the Elliptic dilogarithms}
Recall the {\it Bloch-Wigner function}
\[
D(x):=\mathrm{Im}(\mathrm{ln}_2(x))+\log|x|\mathrm{arg}(1-x).
\]
For $q\ne0$, the {\it elliptic dilogarithms} is defined to be
\[
D_q(x):=\sum_{n\in \Z}D(xq^n),
\]
satisfies $D_q(qx)=D_q(x)$ and $D_q(x^{-1})=-D_q(x)$ (\cite{Ir}, \cite{GL}). 

\medskip

Recall Bloch's formla which decsribes the real regulator via the elliptic dilogarithm
(cf. \cite{GL} p.416--417).
Noting that the $K_2$-symbols \eqref{bei-thm1-eq0} and \eqref{bei-thm1-eq3}
are defined by rational functions supported on torsion points, 
Bloch's formula implies the following.
\begin{thm}
If $-1<t<0$ then
\[
\frac{\pi}{4}
(1-t)^{-\frac{1}{2}}{}_3F_2\left({\frac{1}{2},\frac{1}{2},\frac{1}{2}
\atop 1,\frac{3}{2}};(1-t)^{-1}\right)
=D_q(i)+D_q(iq^{\frac{1}{2}})=D_{q^{\frac{1}{2}}}(i).
\]
If $0<t<1$ then
\[
-\frac{\pi}{8}\left(\log\frac{1-t}{16}+\frac{1-t}{4}
{}_4F_3\left({\frac{3}{2},\frac{3}{2},1,1
\atop 2,2,2};1-t\right)\right)
=D_q(i)+D_q(iq^{\frac{1}{2}}).
\]
where $i=\sqrt{-1}$ and we put\[
q:=\exp
\left(-2\pi\frac{F(\frac{1}{2},\frac{1}{2},1;1-t)}
{F(\frac{1}{2},\frac{1}{2},1;t)}\right).
\]
\end{thm}

\begin{thm}
If $1<t<2$ then
\begin{align*}
&B\left(\frac{1}{3},\frac{1}{3}\right)
t^{-\frac{1}{3}}{}_3F_2\left({\frac{1}{3},\frac{1}{3},\frac{1}{3}
\atop \frac{2}{3},\frac{4}{3}};t^{-1}\right)
+\frac{1}{2}B\left(\frac{2}{3},\frac{2}{3}\right)
t^{-\frac{2}{3}}{}_3F_2\left({\frac{2}{3},\frac{2}{3},\frac{2}{3}
\atop \frac{4}{3},\frac{5}{3}};t^{-1}\right)\\
&=6\sqrt{3}D_q(e^{2\pi i/3})
\end{align*}
where \[
q:=\exp
\left(\frac{-2\pi}{\sqrt{3}}\frac{F(\frac{1}{3},\frac{2}{3},1;t)}
{F(\frac{1}{3},\frac{2}{3},1;1-t)}\right).
\]
\end{thm}

\end{document}